\documentclass[11pt]{amsart}
\usepackage{amsxtra,amscd, color}
\usepackage{amssymb}
\usepackage{enumitem}
\theoremstyle{plain}

\newtheorem{theorem}{Theorem}[section]
\newtheorem{lemma}[theorem]{Lemma}
\newtheorem{definition-theorem}[theorem]{Definition-Theorem}
\newtheorem{proposition}[theorem]{Proposition}

\newtheorem{question}[theorem]{Question}
\newtheorem{corollary}[theorem]{Corollary}
\newtheorem{definition}[theorem]{Definition}
\newtheorem{example}[theorem]{Example}

\newtheorem{remark}[theorem]{Remark}
\newtheorem{remarks}[theorem]{Remarks}
\newtheorem{conjecture}[theorem]{Conjecture}

\newtheorem{Thm}{Theorem}[section]
\newtheorem{Prop}[Thm]{Proposition}
\newtheorem{Lem}[Thm]{Lemma}
\newtheorem{Cor}[Thm]{Corollary}

\theoremstyle{definition} 
\newtheorem{Def}[Thm]{Definition}
\newtheorem{Ex}[Thm]{Example}
\newtheorem{Exs}[Thm]{Examples}
\newtheorem{Rmk}[Thm]{Remark}
\newtheorem{Rmks}[Thm]{Remarks}
\newtheorem{Qtn}[Thm]{Question}
\newtheorem{Qtns}[Thm]{Questions}

\def\ker{{\rm ker \,}}

\def\char{{\rm char \,}}

\def\co{{\rm co \,}}
\def\mm{\mathbf{m}}

\def\OO{\mathcal{O}}

\def\gg{\mathfrak{g}}

\newcommand \bth[1] { \begin{theorem}\label{t#1} }
\newcommand \ble[1] { \begin{lemma}\label{l#1} }

\newcommand \bpr[1] { \begin{proposition}\label{p#1} }
\newcommand \bqu[1] { \begin{question}\label{q#1} }
\newcommand \bco[1] { \begin{corollary}\label{c#1} }
\newcommand \bde[1] { \begin{definition}\label{d#1}\rm }
\newcommand \bex[1] { \begin{example}\label{e#1}\rm }
\newcommand \bre[1] { \begin{remark}\label{r#1}\rm }
\newcommand \bres[1] { \begin{remarks}\label{r#1}\rm }
\newcommand \bcj[1] { \begin{conjecture}\label{j#1}\rm }
\renewcommand {\eth} { \end{theorem} }
\newcommand {\ele} { \end{lemma} }

\newcommand {\epr} { \end{proposition} }
\newcommand {\equ} {\end{question} }
\newcommand {\eco} { \end{corollary} }
\newcommand {\ede} { \end{definition} }
\newcommand {\eex} { \end{example} }
\newcommand {\ere} { \end{remark} }
\newcommand {\eres} { \end{remarks} }
\newcommand {\ecj} { \end{conjecture} }
\newcommand {\enota} { \end{notation} }

\def \OO {{\mathcal{O}}}

\def \mm  {\mathfrak{m}}

\DeclareMathOperator \maxspec { {\mathrm{Maxspec}}}

\DeclareMathOperator \Ann { {\mathrm{Ann}} }

\DeclareMathOperator \Ext { {\mathrm{Ext}} }
\DeclareMathOperator \GKdim {{\mathrm{GK \, dim}}}
\DeclareMathOperator \Kdim {{\mathrm{K \, dim}}}

\DeclareMathOperator \gldim { {\mathrm{gl.dim}} }
\DeclareMathOperator \prdim { {\mathrm{pr.dim}} }

\begin{document}


\title[PI Hopf algebras]
{Affine commutative-by-finite Hopf algebras}
\author[K. A. Brown]{K. A. Brown}
\address{School of Mathematics and Statistics \\
University of Glasgow \\
Glasgow G12 8QQ, Scotland }
\email{Ken.Brown@glasgow.ac.uk}
\author[M. Couto]{M. Couto}
\thanks{The research of Ken Brown was partially supported by a grant from the Leverhulme Foundation, Emeritus Fellowship EM-2017-081. The PhD research of Miguel Couto was supported by a grant of the Portuguese Foundation for Science and Technology, SFRH/BD/102119/2014.}
\address{School of Mathematics and Statistics\\
University of Glasgow \\
Glasgow G12 8QQ, Scotland }
\email{M.Couto.1@research.glasgow.ac.uk}
\date{}

\begin{abstract} The objects of study in this paper are Hopf algebras $H$ which are finitely generated algebras over an algebraically closed field and are extensions of a commutative Hopf algebra by a finite dimensional Hopf algebra. Basic structural and homological properties are recalled and classes of examples are listed. Bounds are obtained on the dimensions of simple $H$-modules, and the structure of $H$ is shown to be severely constrained when the finite dimensional extension is semisimple and cosemisimple.
\end{abstract}
\maketitle

\section{Introduction}
\subsection{}\label{intro1} A Hopf algebra $H$ over the algebraically closed field $k$ is \emph{affine commutative-by-finite} if it is a finitely generated module over a normal commutative finitely generated Hopf subalgebra $A$. In this context, to say that $A$ is \emph{normal} means that it is closed under the adjoint actions of $H$, see $\S$\ref{PIdefn}. As has long been understood and is recalled in $\S$\ref{PIdefn}, such an algebra $H$ should be viewed as an extension of the affine commutative Hopf subalgebra $A$ by the finite dimensional Hopf algebra $\overline{H}:=H/A^+ H$, where $A^+$ denotes the augmentation ideal of $A$.

\subsection{}\label{intro2} This paper is first of a series in which we treat the class of affine commutative-by-finite Hopf $k$-algebras as a laboratory for testing hypotheses about all Hopf algebras of finite Gelfand-Kirillov dimension.

With that more general aim in mind, we first (in $\S$2) review and organise the known properties of commutative-by-finite Hopf algebras, including finiteness conditions, homological properties and representation theory. Most of the results of this section are not new, but are gathered from a number of sources, for example \cite{Br98}, \cite{BrGood}, \cite{LWZ}, \cite{Sk04}, \cite{WZ}, \cite{WuZh}. Then, in $\S$3, we list and describe many important families of these algebras. The sources here include  \cite{BrZh}, \cite{ConProc}, \cite{GoodZh}, \cite{Jacob}, \cite{Liu}. We also recall an example due to Gelaki and Letzter \cite{GL} of a prime Hopf algebra, finite over its affine centre, which is \emph{not} commutative-by-finite.

The material in $\S \S$\ref{stability} and \ref{prime} is crucial to the later results in the paper, including the theorems stated below. In $\S$\ref{stability} work of Skryabin \cite{Sk04, Sk07} and of Montgomery and Schneider \cite{MS} is applied to the action of $\overline{H}$ on $\maxspec (A)$, focussing on the concept of an $\overline{H}$-\emph{orbit} of maximal ideals. In particular, the $\overline{H}$-\emph{core} $\mathfrak{m}^{(\overline{H})}$ of a maximal ideal $\mm$ of $A$, which features in Theorem \ref{introthm2} below, is defined here as the biggest $\overline{H}$-invariant ideal of $A$ contained in $\mm$. The (finite) set of maximal ideals of $A$ which contain $\mm^{(\overline{H})}$ is - by definition - the $\overline{H}$-\emph{orbit} of $\mm$. $\S$\ref{prime} contains analysis of the action of $\overline{H}$ on the nilradical and minimal primes of $A$, and applies this to study the surprisingly strict relation between (semi)primeness of $H$ and of $A$.

 
Most of the new results are in $\S\S$ 6 and 7. In $\S$7 we focus on the case where the finite dimensional Hopf factor $\overline{H}$ can be chosen to be semisimple and cosemisimple. Results of Etingof, Walton and Skryabin \cite{EW, Sk17} are crucial to the key message about this class of algebras, which is that the underlying noncommutativity is generated by the action of a finite group. Theorem \ref{semi}, the main result of $\S 7$, is rather complex to state; the following is the special case where $H$ is assumed to be prime, avoiding many of the technicalities. 

\begin{Thm}[{Corollary \ref{primesemis}}]\label{introthm1} Let $H$ be a prime commutative-by-finite Hopf algebra with affine commutative normal Hopf subalgebra $A$, such that $\overline{H} = H/A^+H$ is semisimple and cosemisimple.  Then $A$ is a domain, and after replacing $A$ by a larger smooth commutative affine domain $D$ which is a normal left coideal subalgebra of $H$,
\begin{enumerate}
\item $H/D^+ H \cong k\Gamma$ for a finite group $\Gamma$ whose order is a unit in $k$;
\item the adjoint action of $\overline{H}$ on $A$ factors through $k\Gamma$, and $\Gamma$ acts faithfully on $A$ via the adjoint action;
\item There exists another group algebra factor $k\Lambda$ of $\overline{H}$, such that $\Lambda$ acts faithfully on $D$ via the left adjoint action, and $\Gamma$ is a factor of $\Lambda$.
\item Suppose in addition that $H$ is pointed. Then $H$ is a crossed product of $D$ by $k \Gamma$, that is, $H\cong D\#_\sigma k\Gamma$ for some cocycle $\sigma$.
\end{enumerate}
\end{Thm}

The results of $\S$6 concern the dimensions of simple modules. Invariant-theoretic results of Skryabin \cite{Sk10} are used to show that, given an affine commutative-by-finite Hopf algebra $H$ with commutative normal Hopf subalgebra $A$, and a simple $H$-module $V$, there exists an $H$-invariant Frobenius algebra factor of $A$, denoted by $A/\mathfrak{m}^{(\overline{H})}$, which acts faithfully on $V$. The following consequence is a simplified special case of Theorems \ref{bound} and \ref{dims}. Recall that if $R$ is a prime affine noetherian algebra which is a finite module over its centre, then the \emph{PI-degree} of  $R$, denoted by $\mathrm{PI.deg}(R)$, is the maximum dimension over $k$ of the simple $R$-modules, \cite[Theorem I.13.5, Lemma III.1.2]{BrGoodbook}.

\begin{Thm}\label{introthm2} Let $H$, $A$ and $V$ be as above, and assume that $A$ is semiprime (as is the case in characteristic 0, for example) and that $H$ is prime. 
\begin{enumerate}
\item There is an $H$-invariant ideal $\mathfrak{m}^{(\overline{H})}$ of $A$, an invariant of $V$, such that the Frobenius algebra $A/\mathfrak{m}^{(\overline{H})}$ embeds in $V$ as an $A$-module. Hence 
$$ \mathrm{dim}_k(A/\mathfrak{m}^{(\overline{H})}) \leq \mathrm{dim}_k(V) \leq \mathrm{dim}_k (\overline{H}). $$
\item Suppose that $\overline{H}$ is semisimple and cosemisimple. Then, in the notation of Theorem  \ref{intro1},
$$ \mathrm{PI.deg} (H) = |\Gamma|. $$
\item Suppose that $\overline{H}$ is semisimple and cosemisimple and that $H$ is pointed. Then (using again the notation of Theorem \ref{intro1}) there is a maximal ideal $\mathfrak{m}$ of $D$ and a positive integer $\ell$, both depending on $V$, such that
$$ \ell |\Gamma : C_{\Gamma}(\mathfrak{m})| = \mathrm{dim}_k (V) \leq |\Gamma|, $$
where $C_{\Gamma}(\mathfrak{m})$ denotes the centraliser of $\mathfrak{m}$ in $\Gamma$.
\end{enumerate}
\end{Thm}

\subsection{Notation}\label{notation} Throughout this paper $k$ will denote an algebraically closed field, all vector spaces are over $k$ unless stated otherwise, and all unadorned tensor products are over $k$. The \emph{Gelfand-Kirillov dimension} of an algebra $R$ will be denoted by $\GKdim (R)$. For details, refer to \cite{KL}. Recall that, for an affine noetherian algebra $R$ satisfying a polynomial identity, this dimension is always a non-negative integer, coinciding with the classical Krull dimension defined in terms of the maximum length of a chain of prime ideals, \cite[Corollary 10.16]{KL}.  The global dimension of $R$ is denoted $\gldim (R)$, and the projective dimension of an $R$-module $M$ by $\prdim (M)$. For a Hopf algebra $H$ we use the usual notation of $\Delta,\epsilon$ for the coalgebra structure, with $\Delta(h)=\sum h_1\otimes h_2$ for $h \in H$, and we use $S$ to denote its antipode. The augmentation ideal $\ker\epsilon$ of $H$ will be denoted by $H^+$. Given a Hopf surjection $\pi:H\to T$, $H$ is canonically a right (and left) $T$-comodule algebra with coaction $\rho=(id\otimes \pi)\Delta$. The subspace of right coinvariants $\lbrace h\in H: \rho(h)=h\otimes 1 \rbrace$ will be denoted by either $H^{\co \pi}$ or $H^{\co T}$. Unexplained Hopf algebra terminology can be found in \cite{Mont} or \cite{Radbook}, for example.

\bigskip

\section{Basic properties}\label{OurClass}

\subsection{Definition and initial properties}\label{PIdefn}

Recall \cite[\textsection 3.4]{Mont} that a subalgebra $K$ of a Hopf algebra $H$ is \emph{normal} if it is invariant under the left and right adjoint actions of $H$; that is, for all $k\in K$ and $h\in H$, 
$$ ad_l(h)(k)= \sum h_1 k S(h_2) \in K \qquad \text{and} \qquad ad_r(h)(k)= \sum S(h_1)kh_2 \in K. $$

\begin{Def}\label{almost} A Hopf $k$-algebra $H$ is \emph{commutative-by-finite} if it is a finite (left or right) module over a commutative normal Hopf subalgebra $A$.
\end{Def}

\begin{Rmk}\label{normal}
For an affine commutative-by-finite Hopf algebra $H$, it is enough to require that $A$ is normal on one side only. This follows from Lemma \ref{leftrightorbss}(1).
\end{Rmk}

A commutative-by-finite Hopf algebra $H$ is an extension of a commutative Hopf algebra by a finite dimensional Hopf algebra. Making this obvious but fundamental observation precise, we have the following basic facts, with corresponding notation which we shall retain henceforth.

\begin{Thm}\label{properties}
Let $H$ be a commutative-by-finite Hopf algebra, finite over the normal commutative Hopf subalgebra $A$ with augmentation ideal $A^+$.
\begin{enumerate}
\item\label{affnoeth} The following are equivalent:
\begin{enumerate}
\item $H$ is noetherian.
\item $H$ is affine.
\item $A$ is affine.
\item $A$ is noetherian.
\end{enumerate}
\item $A^+ H = HA^+$ is a Hopf ideal of $H$, and 
$$ \overline{H} \; := \; H/A^+ H $$
is a finite dimensional quotient Hopf algebra of $H$.
\item The left (resp. right) adjoint action of $H$ on $A$ factors through $\overline{H}$, so that $A$ is a left (resp. right) $\overline{H}$-module algebra.
\item $H$ satisfies a polynomial identity.
\end{enumerate}
Assume in the rest of the theorem that $H$ satisfies the equivalent conditions of (1). Denote by $\pi: H \longrightarrow \overline{H}$ the Hopf algebra surjection from $H$ to $\overline{H}$ given by (2).
\begin{enumerate}
\item[(5)] The antipode $S$ of $H$ is bijective.
\item[(6)] $ \GKdim(H) = \GKdim(A) = \Kdim(H) = \Kdim(A) < \infty. $
\item[(7)] Assume one of the following hypotheses:
\begin{enumerate}
\item[(i)] $\char k = 0$;
\item[(ii)] $A$ is semiprime;
\item[(iii)] $H$ is pointed;
\item[(iv)] $A$ is central in $H$.
\end{enumerate}
\noindent Then:
\begin{enumerate}
\item[(a)] $A\subseteq H$ is a faithfully flat $\overline{H}$-Galois extension;
\item[(b)] $A$ equals the right and the left $\overline{H}$-coinvariants of the $\overline{H}$-comodule $H$; that is,
$$ H^{\co\pi} \; = \; \,^{\co\pi}H \; = \; A;$$
\item[(c)] $H$ is a finitely generated projective generator as left and right $A$-module;
\item[(d)] $A$ is a left (resp. right) $A$-module direct summand of $H$.
\end{enumerate}
\end{enumerate}
\end{Thm}

\begin{proof}
$(1)$ (c) $\Leftrightarrow$ (d): Molnar's theorem \cite{Mol} ensures that (d) implies (c), and the converse is Hilbert's Basis Theorem.

(d) $\Leftrightarrow$ (a): That (d) implies (a) is clear. The converse follows from \cite{FJ}.

(b) $\Leftrightarrow$ (c): That (c) $\Rightarrow$ (b) is trivial. The converse follows from a generalized version of the Artin-Tate lemma attributed to Small, which states that, if $A\subseteq H$ is any extension of $k$-algebras where $H$ is affine and is a finitely-generated left module over a commutative subalgebra $A$, then $A$ is affine. A proof can be found at \cite[Lemma 1.3]{Sarr}.
\medskip

\noindent $(2)$ Normality of $A$ ensures that $A^+H$ is a Hopf ideal by \cite[Lemma 3.4.2(1)]{Mont}. 
\medskip

\noindent $(3)$ This is clear.
\medskip

\noindent $(4)$ Every $k$-algebra which is a finite module over a commutative subalgebra satisfies a PI by \cite[Corollary 13.1.13(iii)]{McRob}.
\medskip

\noindent $(5)$ Every affine noetherian PI Hopf algebra has bijective antipode by \cite[Corollary 2]{Sk06}. 
\medskip

\noindent $(6)$ The fact that the Gelfand-Kirillov and Krull dimensions of $H$ and $A$ are finite and coincide follows from \cite[Proposition 5.5 and Corollary 10.16]{KL}.
\medskip

\noindent $(7)$ (a) Hypothesis (i) is a particular case of (ii) by \cite[Corollary 9.2.11]{Mont}. Assume (ii). Then $A$ has finite global dimension by \cite[11.6, 11.7]{Water}. Then, by \cite[Theorem 0.3]{WuZh}, together with (6), $H$ is a projective left and right $A$-module. A flat extension of Hopf algebras with bijective antipodes is faithfully flat, \cite[Corollary 2.9]{MW}. Together with (5), this proves right and left faithful flatness. If $H$ is pointed (iii) or $A$ is central in $H$ (iv), then faithful flatness follows from \cite[Theorem 3.2]{Tak} and \cite[Theorem 3.3]{Schn93} respectively. The $\overline{H}$-Galois property follows from faithful flatness, as is shown in the proof of \cite[Proposition 3.4.3]{Mont}.

\noindent (b) is immediate from (a) and \cite[Proposition 3.4.3]{Mont}.

\noindent (c) follows from (a) by \cite[Corollary 2.9]{MW}.

\noindent (d) The left $A$-module $H/A$ is in the category ${_A M}^{H}$, in the notation of \cite{MW}. Hence $H/A$ is left $A$-projective by (c) and \cite[Corollary 2.9]{MW}, so the exact sequence 
$$ 0 \longrightarrow A \longrightarrow H \longrightarrow H/A \longrightarrow 0 $$
of left $A$-modules splits, as required. The argument on the right is identical.
\end{proof}

\begin{Rmks}\label{propremarks} Keep the notation and hypotheses of Theorem \ref{properties}.

\noindent$(1)$ Parts (1), (4), (5) and (6) of the theorem are valid (with the same proofs) without the hypothesis that $A$ is normal in $H$.

\medskip

\noindent$(2)$ (Radford \cite{Rad80}) In general, $H$ is \emph{not} a free $A$-module. For example, let $H=\OO(SL_2(k))$. This commutative Hopf algebra is a finite module over the Hopf subalgebra $A$ generated by the monomials of even degree, but it is not a free $A$-module.

\medskip

\noindent$(3)$ Notwithstanding $(2)$, $H$ is $A$-free when $H$ is pointed \cite{Rad77} or when $A$ contains the coradical of $H$ \cite[Corollary 2.3]{Rad77b}.

\medskip

\noindent$(4)$ A further important setting where $H$ is $A$-free is when $H$ decomposes as a crossed product $H = A \#_{\sigma} \overline{H}$. By a celebrated result of Doi and Takeuchi \cite{DT}, this happens if and only if there is a cleaving map $\gamma:\overline{H} \longrightarrow H$ - that is, $\gamma$ is a convolution invertible right $\overline{H}$-comodule map. Moreover, when the extension $A\subset H$ is $\overline{H}$-Galois, such a cleaving map exists if and only if $A \subset H$ has the \emph{normal basis property}.  For details, see for example \cite[Propositions 7.2.3, 7.2.7, Theorem 8.2.4]{Mont}.

Such a crossed product decomposition of $H$ is guaranteed when $H$ is pointed or when its coradical is contained in $AG(H)$ by \cite[Corollary 4.3]{Schn92}. In particular, in these cases all conclusions of Theorem \ref{properties}(7) hold.

\end{Rmks}

\subsection{Finiteness over the centre}\label{centre} The following remains at present a very natural open question. 

\begin{Qtn}\label{ftecent}  \cite[Question E]{Br98}, \cite[Question C(i)]{Br07} Is every affine or noetherian Hopf algebra satisfying a polynomial identity a finite module over its centre?
\end{Qtn}

However, for the special case of affine commutative-by-finite Hopf algebras, the answer to Question \ref{ftecent} is ``yes''. This follows easily from an important result of Skryabin \cite[Proposition 2.7]{Sk04} on integrality over invariants, as we now show.

\begin{Def}\label{invariant} Let $T$ be a Hopf algebra and $R$ a left $T$-module algebra. The subalgebra of $T$-\emph{invariants} of $R$ is
$$ R^T \; = \; \{ r \in R : t\cdot r = \epsilon(t)r, \, \forall t \in T \}. $$
\end{Def}

We give some details here on the proof of this result by Skryabin, since this statement is not explicitly enunciated in \cite{Sk04}.
\begin{Thm}\cite[Proposition 2.7]{Sk04}\label{finmodinv}
Let $T$ be a finite-dimensional Hopf algebra and $A$ an affine commutative left $T$-module algebra. Then, $A$ is a finitely generated module over $A^T$.
\end{Thm}
\begin{proof}
On one hand, if $\char k > 0$, then clearly $A$ is $\mathbb{Z}$-torsion, so \cite[Proposition 2.7(b)]{Sk04} applies to show that $A$ is integral over $A^T$. If on the other hand $\char k = 0$, then $A$ is semiprime by \cite[Corollary 9.2.11]{Mont}, so that, in the terminology of \cite{Sk04}, $A$ is $T$-reduced, and again \cite[Theorem 2.5, Proposition 2.7(a)]{Sk04} give $A$ integral over $A^T$. Since $A$ is affine, it is a finitely generated $A^T$-module.
\end{proof} 

\begin{Cor}\label{centrethm}
Let $H$ be an affine commutative-by-finite Hopf $k$-algebra, finite over the normal commutative Hopf subalgebra $A$. Then, $H$ is a finitely-generated module over its centre $Z(H)$, which is affine.
\end{Cor}
\begin{proof}
Note first that $Z(H) \cap A = A^{\overline{H}}$. For it is clear that $Z(H) \cap A \subseteq A^{\overline{H}}$. Conversely, if $a \in A^{\overline{H}}$ and $h \in H$ then
$$ ha = \sum h_1 a \epsilon(h_2) = \sum h_1 a S(h_2)h_3 = \sum \epsilon (h_1)a h_2 = ah.$$
By Theorem \ref{finmodinv}, $A$ is a finite-module over $A^{\overline{H}}$. Since $H$ is a finite $A$-module, it is finite $Z(H)$-module. Lastly, $Z(H)$ is affine by the Artin-Tate lemma, \cite[Lemma 13.9.10]{McRob}.
\end{proof}

\subsection{Homological properties and consequences}\label{homprops} The class of affine commutative-by-finite Hopf $k$-algebras includes as subclasses the affine commutative Hopf $k$-algebras and the finite dimensional Hopf $k$-algebras. Both these classes exhibit important homological properties - affine commutative Hopf algebras are Gorenstein, \cite[2.3, Step 1]{Br98}, and in characteristic 0 they have finite global dimension - that is, they are \emph{regular} \cite[11.4, 11.6, 11.7]{Water}; and finite dimensional Hopf algebras are Frobenius \cite[2.1.3]{Mont}. We review here how these features partially extend to the commutative-by-finite setting.

The definitions, for a Hopf $k$-algebra $H$, of the \emph{Auslander Gorenstein, Auslander regular, AS-Gorenstein, AS-regular and GK-Cohen Macaulay} properties can be found, for example, in \cite[3.1, 3.6]{WuZh}. The first and third [resp. second and fourth] of these properties are strengthened versions of the requirement to have finite injective [resp. finite global] dimension.

It's immediate from the definitions that a Hopf algebra which is both AS-Gorenstein and GK-Cohen Macaulay must have its injective dimension equal to its GK-dimension. By results of Wu and Zhang, \cite[Proposition 3.7]{WuZh} and \cite[Lemma 6.1, \textsection 6.2]{BrZh1}, every affine noetherian Hopf algebra satisfying a polynomial identity is AS-Gorenstein, Auslander Gorenstein and GK-Cohen Macaulay. This immediately yields most of the following result. Part (\ref{injhom}) gives an alternative way (for algebras which are finite over their centres) of encoding a stringent noncommutative generalisation of the commutative Gorenstein condition - for the definition of \emph{injective homogeneity}, see \cite{BrH}, \cite{BrM}.

\begin{Thm}\label{injthm} Let $H$ be an affine commutative-by-finite Hopf algebra, finite over the commutative normal Hopf subalgebra $A$. Let $\GKdim(H) = d$. 
\begin{enumerate}
\item $H$ is AS-Gorenstein and Auslander Gorenstein, of injective dimension $d$.
\item $H$ is GK-Cohen Macaulay. 
\item $H$ is left and right GK-pure; that is, every non-zero left or right ideal of $H$ has GK-dimension $d$.
\item\label{injhom} $H$ is injectively homogeneous. As a consequence, $H$ is a Cohen-Macaulay $Z(H)$-module.
\item $H$, $Z(H)$ and $A^H$ each have artinian classical rings of fractions, $Q(H)$, $Q(Z(H))$ and $Q(A^H)$.
\item The regular elements $\mathcal{Z}$ of $Z(H)$ and $\mathcal{A}$ of $A^H$ are also non-zero divisors in $H$, and 
\begin{equation}\label{trio}Q(H) = H[\mathcal{Z}]^{-1} = H[\mathcal{A}]^{-1}.\end{equation}
\item $Q(H)$ is quasi-Frobenius.
\end{enumerate}

\end{Thm}

\begin{proof} Parts (1) and (2) are discussed above, and (3) is an immediate consequence of (2), since a non-zero left or right ideal of $H$ has homological grade 0, by definition of the latter.

By Corollary \ref{centrethm} $H$ is finite over $Z(H)$. That (4) is a consequence of (1) and (2) for algebras which are finite over their centres is shown in \cite[Theorems 5.3 and 4.8]{BrM}.

For (5), it is a standard consequence of Small's theorem \cite[4.1.4]{McRob} that GK-pure noetherian algebras have artinian rings of fractions \cite[6.8.16]{McRob}. The claim that $Q(H)$ exists and is artinian is an immediate consequence of this and (3). The arguments for $Z(H)$ and for $A^{\overline{H}}$ are identical; we deal here with $Z(H)$. As with the proof for $H$, it is enough to prove that $Z(H)$ is GK-pure. By Corollary \ref{centrethm} $H$ is a finitely generated $Z(H)$-module. Let $0 \neq I \vartriangleleft Z(H)$, so $\GKdim_H (HI) = d$ by (3). On the other hand
\begin{equation} \label{hip} \GKdim_{Z(H)}(HI) = \GKdim_H(HI),
\end{equation}
by \cite[Corollary 5.4]{KL}, and 
\begin{equation} \label{hop} \GKdim_{Z(H)}(HI) = \GKdim_{Z(H)} (I),
\end{equation}
since $I \subseteq HI$ and $HI$ is a homomorphic image of a finite direct sum of copies of $I$ as $Z(H)$-module, so that \cite[Proposition 5.1(a),(b)]{KL} applies. Therefore $\GKdim_{Z(H)}(I) = d$ as required.

(6) Let $z \in \mathcal{Z}$. Then $\GKdim(Z(H)/Z(H)z) < d$ by \cite[Proposition 3.15]{KL}. Therefore, if $zh = 0$ for some $h \in H$, the $Z(H)$-module $Z(H)h$ has GK-dimension strictly less than $d$, by \cite[Proposition 5.1(c)]{KL}. As in the proof of (5),  $\GKdim_H(Hh) < d$, so $h =0$ by (3). The same argument works for elements of $\mathcal{A}$. The last two partial quotient rings in (\ref{trio}) therefore exist; since they are clearly artinian, they must equal $Q(H)$.

(7) That $Q(H)$ is quasi-Frobenius is proved in \cite[Theorem 0.2(2)]{WuZh}.
\end{proof}

\subsection{Smooth commutative-by-finite Hopf algebras}\label{smooth}
An affine commutative Hopf algebra has finite global dimension if and only if it has no non-zero nilpotent elements \cite[11.6, 11.7]{Water}; while a finite dimensional Hopf algebra has finite global dimension if and only if it is semisimple, if and only if $\epsilon (\int) \neq 0$, for a left or right integral $\int$, \cite[2.2.1]{Mont}. It's thus natural to look for an easily checked necessary and sufficient criterion for an affine commutative-by-finite Hopf algebra to have finite global dimension. Examples suggest there may be no such simple condition, but sufficient conditions for smoothness are not hard to obtain, as follows.  The hypothesis on $H$ as $A$-module in (3) is presumably redundant - for situations where it is known to be satisfied, see Theorem \ref{properties}(7).

\begin{Prop}\label{gldimcrit} Let $H$ be an affine commutative-by-finite Hopf $k$-algebra, with normal commutative Hopf subalgebra $A$ such that $H$ is a finite $A$-module.
\begin{enumerate}
\item If $A$ has no nonzero nilpotent elements and $\overline{H}$ is semisimple, then $H$ has finite global dimension.
\item If $k$ has characteristic 0 and $\overline{H}$ is semisimple, then $H$ has finite global dimension.
\item If $H$ has finite global dimension and $H$ is $A$-flat then $A$ has finite global dimension, so $A$ has no nonzero nilpotent elements.
\end{enumerate}
\end{Prop}

\begin{proof} $(1)$ Since $A$ is semiprime and affine, by Theorem \ref{properties}(1), it has finite global dimension by \cite[11.6, 11.7]{Water}. By Theorem \ref{properties}(7) $H_A$ is faithfully flat, so the finite resolution by projective $A$-modules of the trivial $A$-module $k$ induces a finite projective resolution by $H$-modules of $H\otimes_A k\cong \overline{H}$. The trivial $H$-module is a direct summand of $\overline{H}$ by semisimplicity of $\overline{H}$, hence  $\prdim_H(k) \leq \prdim_A(k) < \infty$. But $\gldim(H) = \prdim_H (k)$ by \cite[Section 2.4]{LorLor},  so the result follows.

(2) This follows from (1), since $A$ is always reduced in characteristic 0 \cite[11.4]{Water}.

(3) If $\gldim(H) < \infty$, then the finite projective resolution of the trivial $H$-module is also a projective $A$-resolution, as ${_A H}$ is projective by Theorem \ref{properties}(7). So $\gldim (A) < \infty$ by \cite[Section 2.4]{LorLor}. 
\end{proof}

\begin{Rmks}\label{smoothcond} (1) As recalled before the proposition, a semiprime affine commutative Hopf algebra $H$ has finite global dimension. But this fails abysmally to generalise to the commutative-by-finite setting. For example,  the algebras $B = B(n,p_0,\ldots,p_s,q)$ constructed by Goodearl and Zhang in \cite{GoodZh} and discussed in $\S$\ref{GKdim2} below are affine commutative-by-finite Hopf algebras and are domains of GK-dimension $2$, but have infinite global dimension. 

(2) The converses of Propositions \ref{gldimcrit}(1) and (2) are false even when $A$ is central or $H$ is cocommutative. For the case when $A$ is central one can take the quantised enveloping algebra $U_{\epsilon}(\mathfrak{g})$ of any simple Lie algebra $\mathfrak{g}$ at a root of unity $\epsilon$, see $\S$\ref{quant}.

As for the cocommutative case, consider the torsion-free polycyclic group
$$ G := \langle x,y : x^{-1}y^2 x = y^{-2}, \, y^{-1}x^2 y = x^{-2}\rangle $$
discussed in \cite[Lemma 13.3.3]{Pa}. As explained by Passman, $G$ has a normal subgroup $N$ which is free abelian of rank 3, with $G/N$ a Klein 4-group. Moreover, $(\ast)$ $G$ does \emph{not} have any normal abelian subgroup $W$ of finite index with $|G:W|$ prime to 2. Take a field $k$ of characteristic 2 and let $H = kG$. Then $H$ is an affine commutative-by-finite domain by \cite[Theorem 13.4.1]{Pa}, and $\gldim(H) = 3$ by Serre's theorem on finite extensions, \cite[Theorem 10.3.13]{Pa}. The normal Hopf subalgebras of $H$ are just the group algebras $kT$ with $T$ a normal subgroup of $G$. Thus there is no way to present $H$ as a finite extension of a commutative normal Hopf subalgebra $A$, with $\overline{H}$ semisimple, in view of property $(\ast)$ and Maschke's theorem. 
\end{Rmks}

We now gather results from the literature to show that smooth affine commutative-by-finite Hopf algebras share many of the attractive properties of commutative noetherian rings of finite global dimension. For the definition of a \emph{homologically homogeneous (hom. hom.)} ring, see \cite{BrH} or \cite{BrM}.

\begin{Thm}\label{smooththm} Let $H$ be an affine commutative-by-finite Hopf $k$-algebra of finite global dimension $d$.
\begin{enumerate}
\item $H$ is Auslander regular, AS-regular and GK-Cohen Macaulay of $\GKdim(H) = d$.
\item $H$ is homologically homogeneous.
\item $H$ is a finite direct sum of prime rings,
$$ H \; = \bigoplus^{t}_{\ell =1} H_{\ell}$$
with $H_{\ell}$ a prime hom. hom. algebra of dimension $d$ for all $\ell$.
\item $Z(H) = \bigoplus_{\ell = 1}^t Z(H_{\ell})$, where $Z(H_{\ell})$ is an affine integrally closed domain of GK-dimension $d$ for all $\ell$.
\end{enumerate}
\end{Thm}

\begin{proof} (1) This follows from Theorem \ref{injthm}(1),(2) and the definitions of these concepts.

(2), (3) Since $H$ is a noetherian PI ring, this follows from (1) and \cite[Theorem 5.4 and discussion on p. 1013]{StaZh}, or alternatively from (1), Corollary \ref{centrethm}, \cite[Theorem 4.8 and Corollary 5.4]{BrM} and \cite[Theorem 5.3]{BrH2}.

(4) From (3), $Z(H) = \bigoplus_{\ell = 1}^t Z(H_{\ell})$, with each $Z(H_\ell)$ an affine domain, since $H_{\ell}$ is prime and $Z(H_{\ell})$ is a factor of the affine algebra $Z(H)$. Since $H$ is a finite $Z(H)$-module by (2) or Corollary \ref{centrethm}, each $H_\ell$ is a finite module over the image $Z(H_\ell)$ of $Z(H)$ in $H_\ell$.  $\GKdim (Z(H_{\ell})) = d$ by (3) and \cite[Proposition 5.5]{KL}, since $H_{\ell}$ is a finite $Z(H_{\ell})$-module. Finally, the summands $H_\ell$ of $H$ are all hom. hom. by (3), so $Z(H_\ell)$ is integrally closed by \cite[6.1]{BrH2}.
\end{proof}

\subsection{Involutory commutative-by-finite Hopf algebras}\label{involut} Recall that a Hopf algebra $H$ is \emph{involutory} if the antipode $S$ of $H$ satisfies $S^2 = \mathrm{id}_H$. Commutative or cocommutative Hopf algebras are involutory \cite[1.5.12]{Mont}; and when $k$ has characteristic 0 and $H$ is finite dimensional, $H$ is involutory if and only if it is semisimple if and only if it is cosemisimple, \cite[2.6]{LR1} and \cite[Theorems 3 and 4]{LR2}. The following results are effectively special cases of results from \cite{WZ}.

\begin{Prop}\label{inv} Let $H$ be an involutory affine commutative-by-finite Hopf $k$-algebra.
\begin{enumerate}
\item If $\char k = 0$ then $\overline{H}$ is semisimple and $\gldim(H) < \infty$, so Theorem \ref{smooththm} applies to $H$. 
\item Suppose that $\char k = p > 0$ and that either 
\begin{enumerate}
\item[(i)] $H$ is semiprime and $p\nmid \mathrm{PI-degree}(H/P)$ for some minimal prime ideal $P$ of $H$; or 
\item[(ii)] $A$ is semiprime and $p\nmid \mathrm{dim}_k (\overline{H})$. 
\end{enumerate}
Then $\gldim(H) < \infty$, so the conclusions of Theorem \ref{smooththm} apply to $H$. 
\end{enumerate}
\end{Prop}

\begin{proof} (1) When $\mathrm{char}\, k = 0$, $\overline{H}$ is semisimple by the result of Larson and Radford recalled before the proposition, so the result follows from Proposition \ref{gldimcrit}(2).

(2) If (i) holds then this is a special case of \cite[Proposition 3.5]{WZ}. Suppose that (ii) holds. Then $\prdim_H (\overline{H}) < \infty$ by the argument used in the proof of Proposition \ref{gldimcrit}(1), so the conclusion follows from \cite[Lemma 1.6(3)]{WZ}.
\end{proof}


\section{Examples of affine commutative-by-finite Hopf algebras}\label{Examples}

Beyond the commutative and the finite dimensional Hopf algebras, here are some other families of examples, and - in $\S$\ref{nonex} - an important non-example.

\subsection{Enveloping algebras of Lie algebras in positive characteristic}\label{envposchar} Assume in this paragraph that $k$ has positive characteristic $p$. The universal enveloping algebra $U(\gg)$ of a finite-dimensional Lie algebra $\gg$ is a noetherian Hopf algebra and in positive characteristic it is finitely-generated over its centre \cite{Jacob}. When $\gg=\bigoplus_{i=1}^m kx_i$ is restricted, with restriction map $x\mapsto x^{[p]}$, $U(\gg)$ is a free module of finite rank over the central Hopf subalgebra $ A=k\langle x_i^p-x_i^{[p]}: 1\leq i\leq m \rangle $, which is a polynomial algebra on these primitive generators (\cite[Section 2.3]{Jantz} or \cite[Theorem I.13.2(8)]{BrGoodbook}). The Hopf quotient $\overline{H}$ is the restricted enveloping algebra of $\gg$, usually denoted by $u^{[p]}(\gg)$. More generally, for any finite-dimensional $k$-Lie algebra $\gg=\bigoplus_{i=1}^m kx_i$, $U(\gg)$ is a free module of finite rank over a central Hopf subalgebra $A=k\langle y_1,\ldots,y_m \rangle$, where each $y_i$ is a $p$-polynomial in $x_i$, \cite[Proposition 2]{Jacob}. In fact, $A$ is a polynomial algebra on these primitive generators.
These algebras $U(\gg)$ are involutory, being cocommutative, and are smooth domains - by \cite[Theorem XIII.8.2]{CarEil},
$$ \gldim( U(\gg)) = \dim \gg = \GKdim( U(\gg)).$$ 

\subsection{Quantized enveloping algebras and quantized coordinate rings at a root of unity}\label{quant} Quantized enveloping algebras $U_q(\gg)$ of semisimple finite-dimensional Lie algebras $\gg$ are noetherian Hopf algebras, \cite[Section 9.1]{ConProc}, \cite[Section 9.1A]{ChPr}. When $q=\epsilon$ is a primitive $\ell$th root of unity, $U_\epsilon(\gg)$ is a free module of rank $\ell^{\dim\gg}$ over a central Hopf subalgebra $Z_0$, \cite[Corollary and Theorem 19.1]{ConProc}, \cite[Propositions 9.2.7 and 9.2.11]{ChPr}. Its finite-dimensional Hopf quotient $\overline{H}=U_\epsilon(\gg)/Z_0^+U_\epsilon(\gg)$ is the restricted quantized enveloping algebra denoted by $u_\epsilon(\gg)$.

Quantized coordinate rings $\OO_q(G)$ of connected, simply connected, semisimple Lie groups $G$ are noetherian Hopf algebras, \cite[Sections 4.1 and 6.1]{ConLy}. If $q=\epsilon$ is a primitive $\ell$th root of unity, $\OO_\epsilon(G)$ contains a central Hopf subalgebra isomorphic to $\OO(G)$, \cite[Proposition 6.4]{ConLy}, \cite[Theorem III.7.2]{BrGoodbook}, and $\OO_{\epsilon}(G)$ is a free $\OO (G)$-module of rank $\ell^{\dim (G)}$ \cite{BGS}. Better, the extension $\OO(G)\subseteq \OO_\epsilon(G)$ is cleft in the sense of Remark \ref{propremarks}(4) (see \cite[Remark 2.18(b)]{AndruGar}). The finite-dimensional Hopf quotient $\overline{H}=\OO_\epsilon(G)/\OO(G)^+\OO_\epsilon(G)$ is the restricted quantized coordinate ring, sometimes denoted by $o_\epsilon(G)$.

The algebras in these families are thus commutative-by-finite, and are smooth of global dimensions $\dim(\gg)$ and $\dim(G)$ respectively, \cite[Theorem XIII.8.2]{CarEil}, \cite[Theorem 2.8]{BrGood}.

\subsection{Group algebras of finitely-generated abelian-by-finite groups}\label{abfingroup} Let $G$ be a finitely generated group with an abelian normal subgroup $N$ of finite index. Then $N$ is finitely generated, so $G$ is polycyclic-by-finite, and hence $kG$ is an affine cocommutative (and so involutory) Hopf algebra, noetherian by \cite[Corollary 10.2.8]{Pa}. Since $kG$ is a finite $kN$-module, it is commutative-by-finite. Further, $kG$ is cleft, (see Remark \ref{propremarks}(4)), decomposing as crossed products $kG \cong kN \#_\sigma k(G/N)$ \cite[Example 7.1.6]{Mont}. Letting $d$ denote the torsion-free rank of $N$, 
$$ \mathrm{inj.dim}( kG)= d. $$ 
By \cite[Theorem 10.3.13]{Pa}, $kG$ has finite global dimension (necessarily $d$) if and only if $\mathrm{char}\, k = 0$, or $\mathrm{char}\, k = p > 0$ and $G$ has no elements of order $p$.

\subsection{Prime regular affine Hopf algebras of Gelfand-Kirillov dimension 1}\label{GKdim1} These Hopf $k$-algebras were completely classified when $k$ is an algebraically closed field of characteristic 0 by Wu, Liu and Ding in \cite{Wu}, building on \cite{BrZh} and \cite{LWZ}. By a fundamental result of Small, Stafford and Warfield \cite{SSW}, a semiprime affine algebra of GK-dimension one is a finite module over its centre. But in fact more is true for these Hopf algebras - they are all commutative-by-finite, as can be checked  on a case-by-case basis.

With $k$ algebraically closed of characteristic 0, there are 2 finite families and 3 infinite families, as follows.
\begin{enumerate}
\item[(I)] The commutative algebras $k[x]$ and $k[x^{\pm 1}]$.
\item[(II)] A single cocommutative noncommutative example, the group algebra $H = kD$ of the infinite dihedral group $D = \langle a,b: a^2=1, aba=b^{-1} \rangle$. Here, $H$ is finite over the normal commutative Hopf subalgebra $A = k\langle b \rangle$.
\item[(III)] The infinite dimensional Taft algebras $T(n,t,q)=k\langle g,x: g^n=1, xg=qgx \rangle$, where $q$ is a primitive $n$th root of 1 in $k$, with $g$ group-like and $x$ $(1,g^t)$-primitive. The commutative normal Hopf subalgebra $A$ is $k[x^{n'}]$, where $n'=n/\gcd(n,t)$.
\item[(IV)] The generalised Liu algebras $B(n,w,q)$, where $n$ and $w$ are positive integers and $q$ is a primitive $n$th root of 1. We define
$$ B(n,w,q) :=k\langle x^{\pm 1}, g^{\pm 1}, y: x \textit{ central}, yg=qgy, g^n=x^w=1-y^n \rangle, $$
with $x$ and $g$ group-like and $y$ $(1,g)$-primitive. Clearly,  $A:=k[x^{\pm 1}]$ is a central Hopf subalgebra over which $B(n,w,q)$ is free of rank $n^2$.
\item[(V)] Let $m$ and $d$ be positive integers with $(1+m)d$ even, and let $q$ be a primitive $2m$th root of 1 in $k$. The Hopf algebras $D(m,d,q)$ are defined in \cite[Section 4.1]{Wu}. $D(m,d,q)$ is finitely generated over the normal commutative Hopf subalgebra $A=k[x^{\pm 1}]$, \cite[(4.7)]{Wu}.
\end{enumerate} 

The above algebras are all free over their respective normal commutative Hopf subalgebras. Families (I)-(IV) are pointed and decompose as crossed products $H\cong A\#_\sigma \overline{H}$; but $D(m,d,q)$ is not pointed, \cite[Proposition 4.9]{Wu}.
 
Given that these algebras are all regular, Theorem \ref{smooththm} applies to them - in particular they are all hereditary, by (1) of that result. The following questions are now obvious:

\begin{Qtn}\label{dimone} (i) What is the classification corresponding to the above when $k$ is algebraically closed of positive characteristic $p$?

(ii) Can the classification be completed in characteristic 0 if the hypothesis of regularity is omitted?
\end{Qtn}

Regarding (ii), considerable progress is made in \cite{Liu}, including the construction of many examples. However, \emph{all} these new examples are commutative-by-finite, as is explicitly noted in \cite{Liu}. Indeed, Liu conjectures \cite[Conjecture 7.19]{Liu} that, in characteristic 0, every prime affine Hopf $k$-algebra of GK-dimension 1 is commutative-by-finite.

\subsection{Noetherian PI Hopf domains of Gelfand-Kirillov dimension two}\label{GKdim2} Continue to assume that $k$ is algebraically closed of characteristic 0. Let $H$ be a noetherian Hopf $k$-algebra domain with $\GKdim (H) = 2$. These were classified in \cite{GoodZh} under the additional assumption that 
$$ \Ext_H^1({_Hk},{_Hk})\neq 0. \qquad \qquad (\sharp) $$
By \cite[Proposition 3.8(c)]{GoodZh}, $(\sharp)$ is equivalent to $H$ having an infinite dimensional commutative Hopf factor; that is, to the quantum group $H$ containing a one-dimensional classical subgroup. There are 5 classes of such Hopf algebras, as follows.
\begin{enumerate}
\item[(I)] The group algebras over $k$ of $\mathbb{Z}\times \mathbb{Z}$ and $\mathbb{Z} \rtimes \mathbb{Z}$.
\item[(II)] The enveloping algebras of the two 2-dimensional $k$-Lie algebras.
\item[(III)] Hopf algebras $A(n,q)$, for $n \in \mathbb{Z}$ and $q \in k^{\times}$, as algebras the localised quantum plane $k\langle x^{\pm 1},y: xy = qyx \rangle$, with $x$ group-like and $y$ $(1,x^n)$-primitive.
\item[(IV)] Hopf algebras $B(n, p_0, \dots , p_s,q)$, where $s \geq 2$, $n,p_0,\ldots , p_s$ are positive integers with $p_0 \mid n$ and $\{p_i : i \geq 1 \}$ strictly increasing and pairwise relatively prime, and $q$ is a primitive $\ell$th root of 1 where $\ell = (n/p_0)p_1 \ldots p_s$. Then $B(n, p_0, \dots , p_s,q)$ is the subalgebra of the localised quantum plane from (III) generated by $x^{\pm 1}$ with $\{y^{m_i} : 1 \leq i \leq s \}$, where $m_i := \Pi_{j \neq i} p_j.$ This supports a Hopf structure with $x$ group-like and $y^{m_i}$ being $(1,x^{m_i n})$-primitive.
\item[(V)] Hopf algebras $C(n)$ for each integer $n$, $n \geq 2$, where $C(n) := k[y^{\pm 1}][x; (y^n - y)\frac{d}{dy}]$, with $y$ group-like and $x$ being $(y^{n-1},1)$-primitive.
\end{enumerate}

All the algebras in the above list which satisfy a polynomial identity are commutative-by-finite. Namely, these are: 
\begin{enumerate}
\item[$\bullet$] the group algebras in (I). The group algebra of $\mathbb{Z}\rtimes\mathbb{Z}$ is module-finite over the normal commutative Hopf subalgebra $\mathbb{Z}\times 2\mathbb{Z}$.
\item[$\bullet$] the enveloping algebra of the 2-dimensional abelian Lie algebra in (II).
\item[$\bullet$] the algebras $A(n,q)$ from (III) where $q$ is a root of unity. These algebras are module-finite over the normal commutative Hopf subalgebra $A=k\langle (x^l)^{\pm 1}, y^{l'} \rangle$, where $q$ is a primitive $l$th root of 1 and $l'=l/\gcd(n,l)$. 
\item[$\bullet$] the algebras in (IV). It is not hard to see that $A := k\langle (y^{m_i})^{p_i}, x^{\pm \ell} \rangle$ is a normal commutative Hopf subalgebra over which $B(n,p_0,\ldots , p_s,q)$ is a finite module.
\end{enumerate}

As was shown in \cite[Propositions 1.6 and 0.2a]{GoodZh}, all the families have global dimension 2, except for (IV), whose members have infinite global dimension, being free over the coordinate ring $k\langle y^{m_i} : 1 \leq i \leq s \rangle$ of a singular curve.

In \cite{WZZ}, it was shown that not all noetherian Hopf $k$-algebra domains of GK-dimension 2 satisfy $(\sharp)$. More precisely, a key discovery of \cite{WZZ} was the existence of:
\begin{enumerate}
\item[(VI)] an infinite family of noetherian Hopf algebra domains of GK-dimension 2, with $\mathrm{Ext}_H^1({_Hk}, {_Hk}) = 0$ for all members of the family. 
\end{enumerate}
The algebras in (VI) are commutative-by-finite - by \cite[Theorem 2.7]{WZZ}, each of them is a finite-rank free module over a central Hopf subalgebra which is the coordinate ring of a 2-dimensional solvable group. Almost all of them have infinite global dimension, some of them being finite-rank free modules over an algebra in (IV).

It is conjectured in \cite[Introduction]{WZZ} that, when $k$ has characteristic 0, the families (I)-(VI) constitute \emph{all} the affine Hopf $k$-algebra domains of GK-dimension 2, at least in the pointed case. Namely, the authors ask:

\begin{Qtn}\label{GKtwoqn}{\rm (Wang, Zhang, Zhuang, \cite{WZZ})} Let $k$ be algebraically closed of characteristic 0, and let $H$ be an affine Hopf $k$-algebra domain of GK-dimension two. If $H$ is pointed, does it belong to the families (I)-(VI)?
\end{Qtn}

At present every known affine Hopf $k$-algebra domain of GK-dimension two is generated by group-like and skew primitive elements. The answer to Question \ref{GKtwoqn} is affirmative for Hopf algebras so generated, \cite[Corollary 0.2]{WZZ}.

\subsection{Affine PI but not commutative-by-finite example}\label{nonex} Gelaki and Letzter gave an example  \cite{GL} of a prime noetherian Hopf $k$-algebra $U$ of Gelfand-Kirillov dimension 2 which is a finite module over its centre, but which is \emph{not} commutative-by-finite. Their example is the bosonisation of the enveloping algebra of the Lie superalgebra $\mathrm{pl}(1,1)$. It is a finite module over its centre, and so is affine PI. But $U$ contains a nonzero element $u$ forming part of a PBW basis of $U$, with $u^2 = 0$. Thus $U$ is not a domain, so does not feature in the list in $\S$\ref{GKdim2}. Moreover, $U$ is a free $k\langle u \rangle$-module, so that 
$$ \prdim_{U}(k) \geq \prdim_{k\langle u \rangle}(k) = \infty,$$
forcing $\gldim(U) = \infty$. Thus, the following question remains open at present:

\begin{Qtn}\label{regqn} Is every affine noetherian regular PI Hopf algebra commutative-by-finite?
\end{Qtn}


\section{$\overline{H}$-Stability and orbital semisimplicity}\label{stability}

We study here the action of $\overline{H}$ on the ideals of $A$. It is convenient in $\S$\ref{orbitsec} and $\S$\ref{sectionorbss} to work in a broader context, returning to commutative-by-finite Hopf algebras in $\S$\ref{stabsec}.

\subsection{Hopf orbits in $\maxspec(A)$}\label{orbitsec}

\begin{Def}\label{stabledef}
Let $T$ be a Hopf algebra, $R$ a left $T$-module algebra. 
\begin{enumerate}
\item A subspace $V$ of $R$ is $T$-\emph{stable} if $t\cdot v \in V$ for all $t \in T$ and $v \in V$.
\item The $T$-\emph{core} of an ideal $I$ of $R$ is the ideal $ I^{(T)} \; = \; \{ x \in I : t\cdot x \in I, \,  \forall t \in T \}.$
\end{enumerate}
\end{Def}

Note that the $T$-core $I^{(T)}$ of $I$ will in general be larger than the subspace $R^T \cap I$ of $T$-invariants of $I$ - hence the use of brackets in our notation. The following lemma is easily checked.

\begin{Lem}\label{stablecore}
With the above notation, $I^{(T)}$ is the largest $T$-stable subspace of $R$ contained in $I$.
\end{Lem}

Let $A$ be a commutative left $T$-module algebra. The notion of the $T$-core of an ideal of $A$ leads to an equivalence relation on the prime spectrum of $A$, as discussed in a more general setting by Skryabin in \cite{Sk10}. However we limit attention here to the case of the maximal ideals of a commutative $T$-module algebra $A$, although everything could be done under weaker hypotheses, as in \cite{Sk10}.

\begin{Def}\label{deforbit}
Let $T$ be a Hopf algebra and $A$ a commutative left $T$-module algebra.
\begin{enumerate}
\item Define an equivalence relation $\sim^{(T)}$ on $\maxspec(A)$ as follows: for $\mm,\mm'\in\maxspec(A)$, $\mm\sim^{(T)} \mm'$ if and only if $ \mm^{(T)}=\mm'^{(T)}.$

\item For each $\mm\in\maxspec(A)$, its $T$-\textit{orbit} is the set of maximal ideals with the same $T$-core, $$\OO_{\mm} = \lbrace \mm'\in\maxspec(A): {\mm'}^{(T)} = \mm^{(T)} \rbrace. $$
\end{enumerate}
\end{Def}

The following key result is essentially due to Skryabin \cite[Theorems 1.1, 1.3]{Sk10}. For affine commutative-by-finite Hopf algebras, we shall in due course (Theorem \ref{bound}) obtain an upper bound for the dimensions of the finite dimensional commutative algebras $A/\mm^{(\overline{H})}$ appearing below, and hence for the cardinalities of the orbits $\mathcal{O}_{\mm}$.

\begin{Prop}\label{Skryab}
Let $T$ be a finite-dimensional Hopf algebra, $A$ an affine commutative left $T$-module algebra and $\mm\in\maxspec(A)$.
\begin{enumerate}
\item $A/\mm^{(T)}$ is a $T$-simple algebra. That is, the only $T$-stable ideals of $A/\mm^{(T)}$ are the trivial ones.

\item $A/\mm^{(T)}$ is a Frobenius algebra. In particular,  $A/\mm^{(T)}$ is finite dimensional.

\item $\OO_{\mm} = \lbrace \mm'\in\maxspec(A): \mm^{(T)}\subseteq \mm' \rbrace$.

\item $\OO_{\mm}$ is finite.
\end{enumerate}
\end{Prop}
\begin{proof}

(1) This is a special case of \cite[Proposition 3.5]{Sk10}.
\medskip

\noindent (2) Since $A/\mm = k$, then $A^T/(\mm \cap A^T) = k$. Thus, by Theorem \ref{finmodinv}, $A/(\mm \cap A^T)A$ is a finite dimensional algebra, and therefore so is its factor algebra $A/\mm^{(T)}$. In particular, being artinian, it equals its classical ring of quotients:
\begin{equation}\label{quotient} A/\mm^{(T)}  = Q(A/\mm^{(T)}).
\end{equation}
Moreover, by (1),  $A/\mm^{(T)}$ is $T$-semiprime.  Thus, by \cite[Theorem 1.3]{Sk10}, $Q(A/\mm^{(T)})$ is quasi-Frobenius. The result follows from this and (\ref{quotient}), noting that a commutative quasi-Frobenius $k$-algebra is actually Frobenius.
\medskip

\noindent (3) If $\mm'\in\OO_\mm$, $\mm^{(T)} = {\mm'}^{(T)} \subseteq \mm'$ and so $\OO_{\mm} \subseteq \lbrace \mm'\in\maxspec(A): \mm^{(T)}\subseteq \mm' \rbrace$. For the reverse inclusion, suppose some $\mm'\in\maxspec(A)$ contains $\mm^{(T)}$, so that $\mm^{(T)} \subseteq {\mm'}^{(T)}$. Hence, by (1), ${\mm'}^{(T)} = \mm^{(T)}$; that is, $\mm'\in \OO_\mm$.
\medskip

\noindent (4) is an immediate consequence of (2) and (3).
\end{proof}

\subsection{Orbital semisimplicity}\label{sectionorbss}
 
When the Hopf algebra $T$ of Definition \ref{deforbit} is a group algebra $kG$ of a finite group $G$, the setting is the familiar one of classical invariant theory. In particular, $G$ acts by $k$-algebra automorphisms on $A$ and for $\mm \in \maxspec(A)$, $\mathcal{O}_{\mm} = \{\mm^g : g \in G \}$ and $\mm^{(kG)} = \bigcap\{\mm^g : g \in G \}$, so that $A/\mm^{(kG)}$ is a finite direct sum of copies of $k$. We isolate this desirable state of affairs in the following definition. 

\begin{Def}\label{orbsemi}
Let $T$ be a finite dimensional Hopf algebra and $A$ an affine commutative left $T$-module algebra. Then, $A$ is $T$-\emph{orbitally semisimple} if $A/\mm^{(T)}$ is semisimple for every $\mm\in\maxspec(A)$. The prefix $T$ will be omitted when this is clear from the context.
\end{Def}

In view of Proposition \ref{Skryab}(3), $A$ is $T$-orbitally semisimple if and only if 
\begin{equation}\label{equisemi} \mm^{(T)} = \bigcap_{\mm'\in\OO_\mm} \mm'
\end{equation}
for every $\mm\in\maxspec(A)$.

At first glance the following result of Montgomery and Schneider, \cite[Theorem 3.7, Corollary 3.9]{MS}, (generalising an earlier result of Chin \cite[Lemma 2.2]{Chin} and specialised to the present setting), might suggest that orbital semisimplicity always holds, at least when $T$ is pointed. However this is not the case, as the example which follows illustrates. 

\begin{Prop}\label{pointed}
Let $T$ be a finite-dimensional Hopf algebra and $A$ an affine commutative $T$-module algebra. Let $T_0\subset T_1\subset \ldots \subset T_m=T$ be the coradical filtration of $T$, and let $\langle T_0 \rangle$ be the Hopf subalgebra of $T$ generated by $T_0$. Then,
\begin{enumerate}
\item For an ideal $I$ of $A$, $ (I^{(\langle T_0 \rangle)})^{m+1} \subseteq I^{(T)}. $

\item For all $\mm,\mm'\in\maxspec(A)$, $\mm^{(T)}=\mm'^{(T)}$ if and only if $\mm^{(\langle T_0 \rangle)}=\mm'^{(\langle T_0 \rangle)}$.

\item Suppose $T$ is pointed, so that $T_0 = \langle T_0 \rangle$. For any $\mm,\mm'\in\maxspec(A)$, $\mm^{(T)}=\mm'^{(T)}$ if and only if $\mm'=g\cdot \mm$ for some group-like element $g$ of $T$.
\end{enumerate}
\end{Prop}

\begin{Ex}\label{Taft} Fix an integer $n$, $n \geq 2$, and let $q$ be a primitive $n$th root of 1 in $k$. Let $T$ be the $n$-dimensional Taft algebra, 
$$ T: = k\langle g,x: g^n=1,x^n=0,xg=qgx\rangle, $$
with $g$ group-like and $x$ $(1,g)$-primitive. Let $A$ be the polynomial algebra $k[u,v]$.

As shown by Allman, \cite[$\S$3]{All}, $A$ is a left $T$-module algebra with the action defined by
$$ g\cdot u=u, \qquad g\cdot v=qv, \qquad x\cdot u=0, \qquad x\cdot v=u. $$ 
This action is not orbitally semisimple. Indeed, one can easily check with the aid of Proposition \ref{pointed}(3) that, for $a \in k^\times$ and  $\mm:=\langle u-a,v \rangle$,
$$ \mm^{(T)} = \langle u-a, v^n \rangle. $$
\end{Ex}
\medskip

In the positive direction, we have the following result. 

\begin{Thm}\label{orbhold}
Let $T$ be a finite dimensional Hopf algebra and $A$ an affine commutative $T$-module algebra. Then, $A$ is $T$-orbitally semisimple in each of the following cases:
\begin{enumerate}
\item the action is trivial;
\item the action factors through a group;
\item $T$ is cosemisimple;
\item $T$ is involutory and $\char k = 0$ or $\char k  = p > \dim_k (A/\mm^{(T)})$ for all $\mm\in\maxspec(A)$.
\end{enumerate}
\end{Thm}

\begin{proof}
(1) and (2) are clear.

\medskip 

\noindent (3) Let $\mm \in \maxspec(A)$, so that $A/\mm^{(T)}$ is a $T$-module algebra, and is $T$-simple by Proposition \ref{Skryab}(1). Thus $A/\mm^{(T)}$ is semisimple by \cite[Theorem 0.5(ii)]{SvO}.

\medskip
\noindent (4) If $k$ has characteristic 0 this is simply a restatement of (3), since in this case $T$ is involutory if and only if it is cosemisimple if and only if it is semisimple, \cite[2.6]{LR1} and \cite[Theorems 3 and 4]{LR2}. Suppose that $\char k  = p > 0$ and let $\mm \in \maxspec(A)$. If $\dim_k(A/\mm^{(T)}) < p$ then its Jacobson radical is $T$-stable, by \cite[Theorem]{Lin}. By the $T$-simplicity of $A/\mm^{(T)}$ ensured by Proposition \ref{Skryab}(1), this forces $A/\mm^{(T)}$ to be semisimple.
\end{proof}

\begin{Rmk} There is a considerable overlap between cases (2), (3) and (4) of the above result. As noted in the proof, being semisimple, cosemisimple and involutory are equivalent conditions on $T$ when $k$ has characteristic 0. Moreover, when $A$ is a commutative $T$-module domain and $T$ is cosemisimple, Skryabin, in \cite[Theorem 2]{Sk17}, showed that the action factors through a group, extending earlier work of Etingof and Walton \cite{EW}.

\end{Rmk}

\subsection{$\overline{H}$-stability}\label{stabsec} The equivalence relation of Definition \ref{deforbit} was studied by Montgomery and Schneider \cite{MS}, before Skryabin \cite{Sk10}, in the special setting of a faithfully flat $T$-Galois extension $R \subset S$. In fact, the equivalence relation as defined on $\mathrm{Spec}(R)$ in \cite[Definition 2.3(2)]{MS} is different from the one given above, defining instead an ideal $I$ of $R$ to be $T$-stable if $IS = SI$, and then using this to define an equivalence relation. The next lemma examines the relation between these two notions of stability.

\begin{Lem}\label{equivstable}
Let $U$ be a Hopf algebra with bijective antipode and $W$ a normal  Hopf subalgebra of $U$. Write $\overline{U} := U/W^+ U$. Let $I$ be an ideal of $W$.  
\begin{enumerate}
\item If $ad_l(u)(x)\in I$ for all $u \in U, x\in I$, then $UI\subseteq IU$, and $IU$ is an ideal of $U$.
\item If $ad_r(u)(x)\in I$ for all $u \in U, x\in I$, then $IU\subseteq UI$, and $UI$ is an ideal of $U$.
\end{enumerate}
If the extension $W\subseteq U$ is faithfully flat $\overline{U}$-Galois, the following are equivalent:
\begin{enumerate}[label=(\roman*)]
\item $ad_l(u)(x)\in I$, for all $u\in U, x\in I$.
\item $ad_r(u)(x)\in I$, for all $u\in U, x\in I$.
\item $UI=IU$.
\end{enumerate}
\end{Lem}
\begin{proof}
If $I$ is invariant under left adjoint action, then for all $x\in I, u \in U$ we have $ux= \sum u_1xS(u_2)u_3 = \sum ad_l(u_1)(x)u_2\in IU$, proving (1). $(2)$ is proved similarly.

Suppose in addition that the extension $W\subseteq U$ is faithfully flat $\overline{U}$-Galois. This implies that, if (i) or (ii) holds, then $UI=IU$ by \cite[Remark 1.2(ii)]{MS}. For the converse, suppose that $IU = UI$, and let $x \in I$ and $u \in U$. Then
$$ ad_l(u)(x)=\sum u_1xS(u_2)\in IU\cap W, $$ 
since $W$ is normal. But the extension $W  \subseteq U$ being faithfully flat gives $IU \cap W=I$. This proves $(iii)\Longrightarrow(i)$ and the proof is analogous for the right adjoint action.
\end{proof}

Returning to our primary focus, let $H$ again be an affine commutative-by-finite Hopf $k$-algebra, with $k$ algebraically closed and $A$ a commutative normal Hopf subalgebra over which $H$ is a finite module. Let $\overline{H} = H/A^+ H$ and let $\pi$ be the Hopf surjection from $H$ to $\overline{H}$ as in Theorem \ref{properties}. Recall from Theorem \ref{properties}(3) that the adjoint actions of $H$ on $A$ factor over $A^+ H$, so that $A$ is a left and right $\overline{H}$-module algebra.

Using the terminology introduced at Definition \ref{stabledef}(1), the \emph{left} $\overline{H}$-core of an ideal $I$ of $A$ with respect to the left adjoint action will be denoted by ${^{(\overline{H})} I}$ and we will define $I$ as \emph{left} $\overline{H}$-\emph{stable} if it is invariant under left adjoint action. \emph{Right} $\overline{H}$-cores are denoted by $I^{(\overline{H})}$ and \emph{right} $\overline{H}$-stable ideals are defined analogously. When $A$ is semiprime or $H$ is pointed, recall from Theorem \ref{properties}(7) that $A \subseteq H$ is a faithfully flat $\overline{H}$-Galois extension, hence left and right stability of ideals are equivalent by Lemma \ref{equivstable} and we refer to $I$ simply as $\overline{H}$-stable. The non-semiprime case will be dealt with in the next section.

In view of Definition \ref{orbsemi}, we say $A$ is \emph{left} orbitally semisimple if $A/{^{(\overline{H})}\mm}$ is semisimple for every $\mm\in\maxspec(A)$. One similarly obtains a notion of \emph{right} orbital semisimplicity.

\begin{Lem}\label{leftrightorbss}
Let $H$ be an affine commutative-by-finite Hopf algebra, with commutative normal Hopf subalgebra $A$.
\begin{enumerate}
\item Let $V$ be a subspace of $A$. Then, $$ S\left(^{(\overline{H})}V\right) = S(V)^{(\overline{H})}. $$
\item $A$ is left orbitally semisimple if and only if it is right orbitally semisimple.
\item Let $I$ be an ideal of $A$ such that $S(I)=I$. Then, $I$ is left $\overline{H}$-stable if and only if it is right $\overline{H}$-stable.
\end{enumerate}
\end{Lem}
\begin{proof}
(1) Let $v\in V$ and $h\in H$. Then,
\begin{eqnarray*}
ad_r(h)(Sv) &=& \sum S(h_1)S(v)h_2 = S\left(\sum S^{-1}(h_2) v h_1 \right) \\
&=& S\left( \sum S^{-1}(h)_1 v S(S^{-1}(h)_2) \right) = S(ad_l(S^{-1}h)(v)).
\end{eqnarray*}
The statement now easily follows from this.

\medskip

\noindent (2) By (1), $S$ induces an isomorphism between $A/{^{(\overline{H})}\mm}$ and $A/S(\mm)^{(\overline{H})}$. Considering $S$ acts on $\maxspec(A)$, the statement follows.

\medskip
\noindent (3) Suppose $S(I)=I$. If $I$ is left $\overline{H}$-stable, then $I^{(\overline{H})} = S(^{(\overline{H})} I) = S(I) = I$. Conversely, if $I$ is right $\overline{H}$-stable, ${^{(\overline{H})}}I = S^{-1}(I^{(\overline{H})}) = S^{-1}(I) = I$.
\end{proof}

Given part (2) of the previous lemma, the adjectives left and right will be omitted from orbital semisimplicity.

In view of Example \ref{Taft} it seems likely that not all affine commutative-by-finite Hopf algebras are orbitally semisimple, but we know of no example at present. All the affine commutative-by-finite Hopf algebras described in $\S$\ref{Examples} satisfy orbital semisimplicity. The Hopf algebras in $\S$\ref{envposchar}, $\S$\ref{quant} and the family (IV) in $\S$\ref{GKdim1} are orbitally semisimple by Theorem \ref{orbhold}(1), since their corresponding normal commutative Hopf subalgebras $A$ are central. The group algebras in $\S$\ref{abfingroup} (including families (II) in $\S$\ref{GKdim1} and (I) in $\S$\ref{GKdim2}) are of course special cases of Theorem \ref{orbhold}(2); and one may check directly that orbital semisimplicity for the families (III),(V) in $\S$\ref{GKdim1} and (III),(IV) in $\S$\ref{GKdim2} follows from Theorem \ref{orbhold}(2).

\section{Prime and semiprime commutative-by-finite Hopf algebras}\label{prime}

We study here the primeness and semiprimeness of commutative-by-finite Hopf algebras, starting with the classical commutative case in $\S$\ref{nil} and moving on to the general case in $\S$\ref{prime2}. The proof of the following well-known lemma is omitted.
\begin{Lem}\label{powers}
Let $I$ be a Hopf ideal of the Hopf algebra $H$. Then $\bigcap_{n=1}^{\infty} I^n$ is a Hopf ideal.
\end{Lem}

\subsection{The nilradical and primeness: commutative case}\label{nil} The proof of the next lemma is included here to facilitate later discussion of noncommutative versions. 

\begin{Lem}\label{commfacts}
Let $A$ be an affine commutative Hopf algebra. 
\begin{enumerate}
\item The nilradical $N(A)$ of $A$ is a Hopf ideal.
\item The left coideal subalgebra $C := A^{\co A/N(A)}$ of $A$ has the following properties:
\begin{enumerate}
\item[(i)] $C$ is a local Frobenius subalgebra of $A$ with $C^+ A \subseteq N(A)$.
\item[(ii)] $A$ is a free $C$-module.
\end{enumerate}
\item There is a unique minimal prime ideal $P$ of $A$ for which $P \subseteq A^+$;
\item $P = \bigcap_n (A^+)^n + N(A)$ is a Hopf ideal.
\end{enumerate}
\noindent Let $T$ be a Hopf algebra such that $A$ is a left $T$-module algebra. Suppose that $N(A)$ is $T$-stable. Then:
\begin{enumerate}
\item[(5)] the prime ideal $P$ is $T$-stable.
\end{enumerate}
\end{Lem}

\begin{proof} (1) By the Nullstellensatz, $N(A) = \bigcap_V \{ \mathrm{Ann}(V) \}$ as $V$ runs through the irreducible $A$-modules. If $V$ and $W$ are irreducible $A$-modules, then they have dimension 1 since $k$ is algebraically closed, so that $V\otimes W$ is also irreducible. Hence $N(A)(V \otimes W) = 0$ and it follows that $\Delta(N(A)) \subseteq N(A) \otimes A + A \otimes N(A).$ Clearly, $S(N(A)) = N(A)$ as $S$ is an automorphism, and $\epsilon (N(A)) = 0$, since $N(A)$ is nilpotent. 

\medskip

\noindent (2) Let $C=A^{\co A/N(A)}$. Clearly this is a left coideal subalgebra of $A$ such that $C^+A\subseteq N(A)$. In particular, $C^+$ is a nilpotent ideal of $C$, hence $C$ is local. A commutative Hopf algebra is flat over its coideal subalgebras by \cite[Theorem 3.4]{MW}, hence $A$ is $C$-flat. Moreover, when checking faithfulness, it suffices to prove $V\otimes_C A\neq 0$ for all simple $C$-modules $V$, \cite[$\S$13.2, Theorem 2]{Water}. Since $C$ is local with unique simple $k$, and $k \otimes_C A \cong A/C^+ A \neq 0$, $A$ is a faithfully flat $C$-module. 

Since $A$ is noetherian, faithful flatness of $A_{C}$ guarantees that $C$ is also noetherian, as a strictly ascending chain of ideals of $C$ would induce a strictly ascending chain of ideals of $A$. In particular, $C^+$ is a finitely generated nilpotent maximal ideal, so that $\dim_k C < \infty$. Since $C$ is finite dimensional local, its only flat modules are free \cite[Corollary 4.3, Theorem 4.44]{Rot}, so $A$ is $C$-free.

Finally, to see that $C$ is Frobenius, note that $A$, being an affine commutative Hopf algebra, is Gorenstein \cite[Proposition 2.3, Step 1]{Br98}. Viewing an injective resolution of $A_{|A}$ as a resolution of $C$-modules by restriction, we see that the free $C$-module $A$ has finite injective dimension. Hence $\mathrm{inj.dim}\, C < \infty$. But commutative Gorenstein finite dimensional rings are self-injective by \cite[Proposition 21.5]{Ei}, thus quasi-Frobenius. Finally, commutative quasi-Frobenius algebras are Frobenius, proving (2).

\medskip

\noindent (3),(4) In view of (1), we may factor by $N(A)$ and so assume that $A$ is semiprime. Hence $A$ has finite global dimension \cite[11.6, 11.7]{Water}, and so is a finite direct sum of domains by \cite[Corollary 10.14]{Ei} and the Chinese Remainder Theorem. In particular, distinct minimal primes of $A$ are comaximal, so there is a unique such prime $P$ contained in $A^+$.

As $P = Ae$ is idempotently generated,  $e = e^n \in (A^+)^n$ for all $n \geq 1$, so that $ P \subseteq (A^+)^n$ 
for all $n$. Since $A/P$ is a commutative affine domain, $\bigcap_n (A^+)^n \subseteq P$ by Krull's Intersection Theorem, \cite[Corollary 5.4]{Ei}. Combining this with the previous inclusion yields 
\begin{equation}\label{charity}  P = \bigcap_{n\geq 1} (A^+)^n.
\end{equation}
But the right hand side of (\ref{charity}) is a Hopf ideal by Lemma \ref{powers}, so (4) is proved. 

\medskip

\noindent (5) Since $N(A)$ is $T$-stable, we can again factor by it in proving the $T$-stability of $P$. Since $\bigcap_n (A^+)^n$ is an intersection of $T$-stable ideals, the result follows from (\ref{charity}).
\end{proof}

\begin{Rmks}\label{group} Three points concerning Lemma \ref{commfacts}: (1) $C$ is not in general a Hopf subalgebra of $A$. For instance, following \cite[Example 2.5.3]{Brion}, let $k$ be a field of characteristic $p>0$ and fix a positive integer $n$. Let $G=(k,+)\rtimes k^\times$, where $k^\times$ acts on $(k,+)$ by multiplication. Its coordinate ring is $\OO(G) = k[x,y^{\pm 1}]$ where $x$ is $(1,y)$-primitive and $y$ is group-like. Let $A$ be the Hopf quotient $\OO(G)/(x^{p^n})$. Its nilradical is $N(A)=\overline{x}A$ and the coinvariants are $C=A^{\co A/N(A)} = k\langle \overline{x} \rangle$.

\noindent (2) Since  $A/N(A)$ is an affine commutative semiprime Hopf $k$-algebra it is the coordinate ring $\mathcal{O}(G)$ of an affine algebraic $k$-group $G$. The ideal $P$ of $A$  is the defining ideal of $G^\circ$, the connected component of the identity of $G$, so $A/P=\OO(G^\circ)$ \cite[$\S$6.7]{Water}.

\noindent (3)  We know of no examples where this inclusion $C^+ A \subseteq N(A)$ of Lemma \ref{commfacts}(2)(i) is strict. The equality $C^+A=N(A)$ is known to hold if the coradical of $A$ is cocommutative (in particular, if $A$ is pointed) by \cite[Theorem]{Mas91}.
\end{Rmks}

\begin{Lem}\label{lemmaB}
Keep the notation from Lemma \ref{commfacts} and  Remark \ref{group}(2). Define $ B:=A^{\co A/P}, $ a left coideal subalgebra of $A$ with $C \subseteq B$, over which $A$ is flat.
\begin{enumerate}
\item $B + N(A)/N(A) =\OO(G/G^\circ)$, the coordinate ring of the \emph{discrete part of $G$};
\item $B + N(A)/N(A)$ is a finite-dimensional semisimple Hopf subalgebra of $A/N(A)$ with $P=B^+A  + N(A)$;
\item $A/N(A)$ is a free $B + N(A)/N(A)$-module.
\item Suppose that $A$ is semiprime or that the coradical of $A$ is cocommutative. Then $B$ is a semilocal (finite dimensional) Frobenius subalgebra of $A$ with $P = B^+ A$, and $A$ is a free $B$-module and $B$ is a free $C$-module.
\end{enumerate}
\end{Lem}
\begin{proof} Flatness of $A$ over $B$ follows from \cite[Theorem 3.4]{MW}. (1) is a basic part of the theory of algebraic groups, see for example \cite[$\S$16.3, Lemma 2]{Water}, and (2) is immediate from (1). Since $G^\circ$ is a normal subgroup of $G$ \cite[$\S$6.7]{Water}, $P/N(A)$ is a conormal Hopf ideal of $A/N(A)$ and $B^+A + N(A)=P$ by \cite[$\S$5]{Tak}. A commutative Hopf algebra is free over its finite dimensional Hopf subalgebras \cite[Section 2, Theorem 1]{Rad5}, proving (3).

\medskip
\noindent (4) Suppose that the coradical of $A$ is cocommutative. Then $C^+ A = N(A)$ by Remark \ref{group}(3). Thus 
$$ C^+ B = N(B) = N(A) \cap B,$$
and so $B/C^+ B$ is semisimple by (2). Since $C^+B$ is a nilpotent finitely generated ideal of $B$, it follows that $B$ is finite dimensional. A commutative Hopf algebra is free over any finite-dimensional left (or right) coideal subalgebra \cite[Theorem 3.5(iii)]{Mas94}, which proves $A$ is $B$-free. In the semiprime case this is (2) and (3).

Moreover, $B = \oplus_{i=1}^t B_i$, where each $B_i \cong C$ is a finite dimensional commutative local algebra. By Lemma \ref{commfacts}(2), this proves $B$ is Frobenius, semilocal and a free $C$-module.

\end{proof}

\begin{Rmk}\label{copout} Presumably (4) is true without the additional hypothesis on the coradical, but we have not been able to prove this. It is not true in general that $B$ is a Hopf subalgebra of $A$, as is shown by the example in Remark \ref{group}(1), in which $C = B$.
\end{Rmk}

\subsection{The nilradical and primeness: commutative-by-finite case}\label{prime2} We carry the observations and notation of the previous subsection into the next result, which provides parallel results for commutative-by-finite Hopf algebras. These are in part motivated by speculations of Lu, Wu and Zhang \cite[$\S$6, Theorem 6.5, Remark 6.6]{LWZ}, proposing that an exact sequence of Hopf algebras similar to that recalled in Lemma \ref{lemmaB} for the coordinate ring $A$ of an affine algebraic group $G$, namely  
$$0 \longrightarrow B = A^{\co A/P} = \mathcal{O}(G/G^{\circ}) \longrightarrow A = \mathcal{O}(G) \longrightarrow A/P = \mathcal{O}(G^{\circ}) \longrightarrow 0, $$
should be valid more widely. They in fact prove a partial version of this suggestion for noetherian affine regular Hopf $k$-algebras of GK-dimension 1, with $k$ of arbitrary characteristic and not necessarily algebraically closed, \cite[Theorem 6.5]{LWZ}.

\begin{Prop}\label{Hprime}
Let $H$ be an affine commutative-by-finite Hopf algebra, with commutative normal Hopf subalgebra $A$. Let $P$ and $C$, $G$ and $B$ be as in Lemma \ref{commfacts}, Remark \ref{group} and Lemma \ref{lemmaB} respectively.
\begin{enumerate}
\item $N(A)$ is left $\overline{H}$-stable if and only if it is right $\overline{H}$-stable.
\end{enumerate}
Assume for the rest of the proposition that $N(A)$ is $\overline{H}$-stable.
\begin{enumerate}\setcounter{enumi}{1}
\item The minimal prime ideal $P$ of $A$ is an $\overline{H}$-stable Hopf ideal, so $N(A)H$ and $PH$ are Hopf ideals of $H$. Moreover, \begin{equation}\label{nilradint} N(A)H\cap A \; = \; N(A) \; = \; N(H) \cap A,
\end{equation}
where $N(H)$ denotes the nilradical of $H$, and $$ PH\cap A = P. $$

\item $C$ is invariant under the left adjoint action of $H$.

\item Let $Q_1, \ldots , Q_r$ be the minimal prime ideals of $H$ which are contained in $H^+$, so $r \geq 1$, with $r = 1$ if $H/N(A)H$ is smooth. For all $i = 1, \ldots , r$, 
$$ Q_i \cap A \quad = \quad P,$$ and 
$$ N(A)H \subseteq PH \subseteq \bigcap_{i=1}^r Q_i .$$

\item If $H$ is prime, then $A$ is a domain.

\item Assume that either $A$ is semiprime or $H$ is pointed.
\begin{itemize}
\item[(i)] The right coinvariants of the coaction induced by the Hopf surjection $H \twoheadrightarrow H/N(A)H$ are $$ H^{\co H/N(A)H} = A^{\co A/N(A)} = C, $$ a local Frobenius left coideal subalgebra of $A$ over which $H$ is a free module and $C^+A = N(A)$.

\item[(ii)] The right coinvariants induced by the Hopf quotient $\widehat{\pi}:H\to H/PH$ are 
$$ H^{\co H/PH} = A^{\co A/P} = B, $$ 
a semilocal Frobenius left coideal subalgebra of $A$ over which $H$ is a faithfully flat projective module and $B^+H = PH$. Moreover, $B/C^+B = \mathcal{O}(G/G^{\circ})$ is a semisimple normal Hopf subalgebra of $H/N(A)H$ over which $H/N(A)H$ is a free module.
\end{itemize}
\end{enumerate}
\end{Prop}
\begin{proof} (1) Since $S_{|A}$ is an automorphism, $S(N(A)) = N(A)$. The statement follows from Lemma \ref{leftrightorbss}(3).

\medskip
\noindent (2) By Lemma \ref{commfacts} $P$ is left and right $\overline{H}$-stable, hence $PH$ and $N(A)H$ are ideals by Lemma \ref{equivstable}. Since both $P$ and $N(A)$ are Hopf ideals of $A$ by Lemma \ref{commfacts}, $PH$ and $N(A)H$ are Hopf ideals of $H$. By (1) and Lemma \ref{equivstable}, $N(A)H=HN(A)$. Hence $N(A)H$ is nilpotent and  so $N(A)H \cap A = N(A)$. Moreover, $N(A)H\subseteq N(H)$, and so $N(A) \subseteq N(H) \cap A$. The reverse inclusion is obvious. The minimal prime ideal $P$ is the annihilator of some nonzero ideal $I$ of $A$ \cite[Theorem 86]{Kap}, hence $I(PH\cap A) = 0$ and so $PH\cap A \subseteq P$. The reverse inclusion is clear.

\medskip
\noindent (3) Consider the Hopf surjection $\pi':A\to A/N(A)$ and denote the corresponding right $A/N(A)$-coaction of $A$ by $\rho$. Take $h\in H$ and $c\in C$. Thus, $$ \rho(h\cdot c) = \sum h_1c_1S(h_4)\otimes \pi'(h_2c_2S(h_3)) = \sum h_1c_1S(h_3)\otimes \pi'(h_2\cdot c_2). $$
Since by Lemma \ref{commfacts}(2) $C$ is a left coideal with $C^+\subseteq N(A)$ and $N(A)$ is left $\overline{H}$-stable, we have $\rho(h\cdot c) = \sum h_1c_1S(h_3)\otimes \epsilon(h_2)\epsilon(c_2) = (h\cdot c) \otimes 1$. Therefore, $C$ is invariant under left adjoint action of $H$.

\medskip
\noindent (4) Let $Q_1, \ldots , Q_r$ be as stated, so that clearly $r \geq 1$, with $r = 1 $ if $H/N(A)H$ is smooth since in this case $H/N(A)H$ is a finite direct sum of prime rings by Theorem \ref{smooththm}(3). Everything to be proved concerns objects which contain $N(A)H$ and by (\ref{nilradint}) $N(A)H \cap A = N(A)$. Thus we can factor by the Hopf ideal $N(A)H$ of $H$, and hence assume in proving (4) that $A$ is semiprime. 

Let $P = P_1, \ldots , P_s$ be the minimal primes of $A$, so that $$ I:= \bigcap_{j=2}^s P_j $$ is an ideal of A which is \emph{not} contained in $A^+$ (noting that $A$, being semiprime, is smooth and hence a direct sum of domains). Now $IP=\{0\}$. Hence, by the left $\overline{H}$-stability of $P$ and Lemma \ref{equivstable}, for each $i = 1, \ldots , r$,
\begin{equation} \label{product} IHP  = \{ 0 \} \subseteq Q_i. 
\end{equation}
But $I$ is not contained in $Q_i$ since $Q_i \cap A \subseteq A^+$, so (\ref{product}) implies that 
\begin{equation}\label{inside} P \subseteq Q_i \cap A. 
\end{equation}

On the other hand, comparing GK-dimensions,
\begin{eqnarray*} \GKdim(A) &\geq &  \GKdim(A/P) \geq \GKdim(A/Q_i \cap A)\\
&=& \GKdim(H/Q_i) = \GKdim(H), 
\end{eqnarray*}
using (\ref{inside}) for the second inequality and \cite[Proposition 5.5]{KL} for the first equality. For the final equality, observe first that, since it is a minimal prime,  $Q_i$ survives as a proper ideal $Q(H)Q_i$ in the artinian quotient ring $Q(H)$ of $H$, which exists by Theorem \ref{injthm}(5). Hence the (left, say) annihilator $J$ of $Q_i$ in $H$ is non-zero by Theorem \ref{injthm}(7), since this is true for all proper ideals in a quasi-Frobenius ring, \cite[Proposition XVI.3.1]{St}. Thus, $\GKdim(J) \leq \GKdim(H/Q_i) \leq \GKdim(H)$ by \cite[Proposition 5.1, Lemma 3.1]{KL}, and so the desired equality follows from Theorem \ref{injthm}(3). Since $\GKdim(A) = \GKdim(H)$ by Theorem \ref{properties}(6), equality holds throughout the above chain of inequalities, so that $\GKdim(A/P) = \GKdim(A/Q_i \cap A)$.  Since any proper quotient of the affine domain $A/P$ must have a strictly lower GK-dimension \cite[Proposition 3.15]{KL}, this proves (4).

\medskip

\noindent (5) This is a special case of (4), with $r = 1$ and $Q_1 = \{0\}$, which forces $P = \{0\}$.

\medskip

\noindent (6) Assume $A$ is semiprime or $H$ is pointed. Recall the Hopf epimorphism $\pi:H \to \overline{H}$. By definition, $H^{\co H/N(A)H} \subseteq H^{\co H/PH} \subseteq H^{\co \pi} = A$, the last equality following from Theorem \ref{properties}(7).

\noindent (i) By (\ref{nilradint}) $H^{\co H/N(A)H} = C$. When $H$ is pointed, $C^+A = N(A)$ by \cite[Theorem]{Mas91} (this is trivial when $A$ is semiprime). By Theorem \ref{properties}(7) and Lemma \ref{commfacts}(2) $H$ is $C$-projective and, since $C$ is local, $H$ is $C$-free \cite[Theorem 4.44]{Rot}.

\noindent (ii) By (2), $H^{\co H/PH} = B$. By Theorem \ref{properties}(7) and Lemma \ref{lemmaB}(4) $H$ is a faithfully flat projective $B$-module. By Lemma \ref{lemmaB}(2) $B/C^+B = \mathcal{O}(G/G^{\circ})$ is a finite dimensional semisimple Hopf subalgebra of $A/N(A)$ and $P = B^+ A$. By (2) and Lemma \ref{equivstable}, $B^+H=HB^+$. Since $H/N(A)H$ is faithfully flat over $B/C^+B$, $B/C^+B$ is normal in $H/N(A)H$ by \cite[Proposition 3.4.3]{Mont}. A Hopf algebra is free over any finite-dimensional normal Hopf subalgebra \cite[Theorem 2.1(2)]{Schn93}, thus $H/N(A)H$ is a free $B/C^+B$-module. 
\end{proof}

\begin{Rmks}\label{primermks} (1) An unsatisfactory aspect of Proposition \ref{Hprime} is the need to assume $N(A)$ is $\overline{H}$-stable. It seems unlikely that this will always hold, and indeed the structure of the nilradical $N(H)$ is a delicate question. It is clear that if $A$ is semiprime and $\overline{H}$ is semisimple then $H$ is semiprime (and is even a direct sum of prime algebras by Proposition \ref{gldimcrit}(1) and Theorem \ref{smooththm}(3)), but the converse is easily seen to be false. And even the question as to when a smash product $A\# \overline{H}$ of a commutative $\overline{H}$-module algebra $A$ by a finite dimensional Hopf algebra $\overline{H}$ is semiprime has been the subject of much research and remains currently unresolved - see for example \cite{SvO}. Notice that even if one considers a finite dimensional commutative $T$-module algebra $R$, with $T$ a finite dimensional Hopf algebra,  then $N(R)$ may not be $T$-stable - for instance, consider Example \ref{Taft} and in its notation take $T$ to be the $n^2$-dimensional Taft algebra and $R$ to be $A/\mm_a^{(T)}$ for some $a\in k^\times$.

\medskip

\noindent (2) Even when $N(A)$ is $\overline{H}$-stable, the exact sequence of Hopf algebras 
$$0 \longrightarrow B \longrightarrow H \longrightarrow H/PH \longrightarrow 0 $$
given by Proposition \ref{Hprime} fails to realise the picture proposed by \cite{LWZ} which was discussed before the proposition. This is because $PH$ is in general \emph{not} a prime ideal of $H$. In the notation of the proposition consider, for instance, the trivial case where $A = k$, so $H = \overline{H}$, $P = \{0\}$, $r = 1$ and $Q_1 = H^+$.
\end{Rmks}

Notwithstanding Remark \ref{primermks}(1), there is no problem when $A$ is orbitally semisimple (Definition \ref{orbsemi} and discussion in $\S$\ref{stabsec}).

\begin{Prop}\label{(H)hold}
Let $H$ be an affine commutative-by-finite Hopf algebra with normal commutative Hopf subalgebra $A$. Suppose $A$ is orbitally semisimple. Then, the nilradical $N(A)$ is a left and right $\overline{H}$-stable Hopf ideal of $A$.
\end{Prop}
\begin{proof} Since $A$ is orbitally semisimple, $$ N(A) = \bigcap_{\mm\in\maxspec(A)} \mm = \bigcap_{\mm\in\maxspec(A)} \mm^{(\overline{H})} = \bigcap_{\mm\in\maxspec(A)} {^{(\overline{H})}\mm} $$ by Lemma \ref{leftrightorbss}(2), and since each ideal $\mm^{(\overline{H})}$ is right $\overline{H}$-stable and each ${^{(\overline{H})}\mm}$ is left $\overline{H}$-stable, $N(A)$ is left and right $\overline{H}$-stable.
\end{proof}

\section{Representation theory: simple modules}\label{rep}

\subsection{Background facts}\label{back} Let $H$ be an affine commutative-by-finite Hopf algebra, and recall that, throughout this paper, the field $k$ is assumed to be algebraically closed. (The results of this section could be recast without this last hypothesis, but they would be significantly more complicated.) Since $H$ is a finite module over its affine centre by Corollary \ref{centrethm}, its simple modules are finite dimensional over $k$, by Kaplansky's theorem \cite[I.13.3]{BrGoodbook}. To be more precise, let $Q_1, \ldots , Q_t$ be the minimal prime ideals of $H$, so every simple $H$-module is annihilated by at least one $Q_i$. Recall Posner's theorem, for example from \cite[I.13.3]{BrGoodbook}, stating that each algebra $H/Q_i$ has a central simple quotient ring $Q(H/Q_i)$, and the \emph{PI-degree} of $H/Q_i$ is defined to be the square root $n_i$ of the dimension of $Q(H/Q_i)$ over its centre. So $n_i$ is a positive integer, and 
$$ \mathrm{max}\{ \mathrm{dim}_k (V) : V \textit{ a simple } H/Q_i\textit{-module}\} = n_i, $$
by \cite[Theorem I.13.5, Lemma III.1.2(2)]{BrGoodbook}. Indeed most simple $H/Q_i$-modules have dimension $n_i$, in that the intersection of the annihilators of these topmost-dimension simple $H/Q_i$-modules is $Q_i$, whereas the intersection of the annihilators of the smaller simple modules strictly contains $Q_i$, \cite[Lemma III.1.2]{BrGoodbook}.

Given the above we define the \emph{representation theoretic PI-degree of} $H$ to be
$$ \mathrm{rep.PI.deg}(H) := \mathrm{max}\{ n_i : 1 \leq i \leq t \}. $$
Thus $ \mathrm{rep.PI.deg}(H) = \mathrm{max}\{\mathrm{dim}_k (V) : V \textit{ a simple } H\textit{-module}\}.$ 

Since the \emph{minimal degree} $\mathrm{min.deg}(R)$ of a ring $R$ satisfying a polynomial identity is by definition the minimal degree of a monic multilinear polynomial satisfied by $R$,
\begin{equation}\label{upper} \mathrm{rep.PI.deg}(H)  \leq \frac{1}{2}\mathrm{min.deg}(H) 
\end{equation}
by \cite[I.13.3]{BrGoodbook}, where this is an equality if $H$ is semiprime, but in general is strict.

\subsection{Bounds on dimensions of simple modules}\label{bounds}

\begin{Thm}\label{bound}
Let $H$ be an affine commutative-by-finite Hopf algebra, finite over the normal commutative Hopf subalgebra $A$, and $V$ a simple left $H$-module. Keep the notation of $\S$\ref{back} and let 
$$  n(V) \; := \; \mathrm{min}\{ n_i : Q_i \cdot V = 0, \, 1 \leq i \leq t \}.$$
Let $d_A(H)$ denote the minimal number of generators of $H$ as an $A$-module.
\begin{enumerate}
\item There exists $\mm\in\maxspec(A)$ such that $\Ann_A(V)=\mm^{(\overline{H})}$.
\item There is an embedding $A/\mm^{(\overline{H})} \hookrightarrow V$ of $A$-modules. Hence,
$$ \mathrm{dim}_k(A/\mm^{(\overline{H})}) \leq \mathrm{dim}_k(V) \leq n(V) \leq \mathrm{rep.PI.deg}(H) \leq \frac{1}{2}\mathrm{min.deg}(H) \leq d_A(H), $$
\end{enumerate}
where the final inequality requires ${_A H}$ to be projective, as ensured by any of hypotheses (i)-(iv) of Theorem \ref{properties}(7).
\end{Thm}
\begin{proof}
\noindent (1) As discussed in $\S$\ref{back},
\begin{equation}\label{trap} \mathrm{dim}_k(V) \leq n(V) \leq \mathrm{rep.PI.deg}(H).
\end{equation}
In particular, $\mathrm{dim}_k(V) < \infty$, so that ${_A V}$ contains a simple $A$-submodule $V_0$ with annihilator $\mm\in\maxspec(A)$. Since $H\mm^{(\overline{H})}$ is a 2-sided ideal of $H$, $\lbrace v\in V: (H\mm^{(\overline{H})})v=0 \rbrace$ is a non-zero $H$-submodule of $V$. By simplicity, $(H\mm^{(\overline{H})})V=0$.
In particular, $\mm^{(\overline{H})}\subseteq \Ann_A(V)$. Conversely, $\Ann_A(V)$ is contained in $\mm$ and is easily seen to be $\overline{H}$-stable, so that $\Ann_A (V) \subseteq \mm^{(\overline{H})}$ by Lemma \ref{stablecore}, proving (1).

\medskip

\noindent (2) Let $\lbrace v_1,\ldots,v_t \rbrace$ be a $k$-basis of $V$. Then $A/\mm^{(\overline{H})}$ embeds in $V^{\oplus t}$ via
$$ 
\iota:A \to V^{\oplus t} \; : \;a \mapsto (av_1,\ldots,av_t),$$
because $\ker\iota = \mm^{(\overline{H})}$ by (1). Since $A/\mm^{(\overline{H})}$ is a Frobenius algebra by Proposition \ref{Skryab}(2), it is self-injective. Therefore, $A/\mm^{(\overline{H})}$ is (isomorphic to) a direct summand of the $A/\mm^{(\overline{H})}$-module $V^{\oplus t}$, thanks to its inclusion in $V^{\oplus t}$ via $\iota$. Now $A/\mm^{(\overline{H})}$ is finite-dimensional, hence commutative artinian, so it is a (finite) direct sum of non-isomorphic indecomposable submodules, each of which must be a summand of ${_AV}$. Therefore, $A/\mm^{(\overline{H})}$ embeds in ${_A V}$.

The first four inequalities in the displayed chain now follow from this embedding together with (\ref{trap}) and (\ref{upper}). To prove the final inequality, assume one of the four hypotheses (i)-(iv) of Theorem \ref{properties}(7), so that $H$ is a projective $A$-module by Theorem \ref{properties}(7)(c). Thus, for every maximal ideal $\mathfrak{n}$ of $A$, $A_{\mathfrak{n}} \otimes_A H$ is a free left $A_{\mathfrak{n}}$-module of rank at most $d_A(H)$. Via the right action of $H$, this yields a homomorphism from $H$ to the algebra of $d_A(H) \times d_A (H)$ matrices over $A_{\mathfrak{n}}$. Since the intersection of the kernels of these maps as $\mathfrak{n}$ ranges through $\maxspec(A)$ is $\{0\}$, it follows from the Amitsur-Levitski theorem, \cite[Theorem 13.3.3(iii)]{McRob}, that $\mathrm{min.deg}(H) \leq 2d_A(H )$, as required. 
\end{proof}

Typically one expects the upper bound $d_A (H)$ in the above to be replaceable by $\mathrm{dim}_k (\overline{H})$. This is certainly the case when $A$ is a domain. This and other simplifications yield the following.

\begin{Cor}\label{primePI} Retain the notation of Theorem \ref{bound}. Suppose that $H$ is a prime affine commutative-by-finite Hopf algebra such that $A$ is semiprime. Let $V$ be a simple $H$-module and $\mm$ the annihilator of a simple $A$-submodule of $V$. Then
\begin{equation}\label{manyineq} \mathrm{dim}_k(A/\mm^{(\overline{H})}) \leq \mathrm{dim}_k(V)  \leq \mathrm{rep.PI.deg}(H) = \frac{1}{2}\mathrm{min.deg}(H) \leq  \mathrm{dim}_k (\overline{H}). \end{equation}
\end{Cor}
\begin{proof} The equality $\mathrm{rep.PI.deg}(H) = \frac{1}{2}\mathrm{min.deg}(H)$ when $H$ is prime follows from \cite[Lemma III.1.2]{BrGoodbook}. Since $A$ is semiprime the hypothesis of Proposition \ref{Hprime} holds, and $A$ is a domain by part (5) of that result. Therefore, ${_A H}$ is projective by Theorem  \ref{properties}(7)(c), so all of Theorem \ref{bound}(2) is valid. Moreover, $_A H$ is a locally free $A$-module of constant rank, and this rank $r$ must be given by
$$ r = \mathrm{rank}_{A_{A^+}}(A_{A^+} \otimes_A H) = \mathrm{dim}_k(\overline{H}),$$
where the equality follows by Nakayama's lemma. Localising further, to the quotient field $Q(A)$ of $A$, we see that $H$ embeds via right multiplication operators in  $r \times r$ matrices over $Q(A)$. Hence $\mathrm{min.deg}(H) \leq 2r$ by Amitsur-Levitski, \cite[Theorem 13.3.3(iii)]{McRob}.
\end{proof}

\section{Commutative-by-(semisimple \& cosemisimple) Hopf algebras}\label{combysemisimple}

\subsection{Preliminaries}\label{coideal} In moving towards a deeper understanding of affine commutative-by-finite Hopf algebras, an obvious strategy is to impose restrictions on the finite dimensional Hopf quotient $\overline{H}$ of such an algebra $H$. Adopting this approach, a natural first class to study are those Hopf $k$-algebras $H$ which are affine and commutative-by-finite with $\overline{H} = H/A^+ H$ semisimple and cosemisimple for some choice of normal commutative Hopf subalgebra $A$. For brevity, we shall write in this case
$$ A \subseteq H \in\mathcal{CSC}(k). $$

Here are two obvious constructions of such Hopf algebras. First, take a coordinate ring $A=\OO(T)$ of an algebraic group $T$ over a field $k$ and a finite group $\Gamma$ whose order is a unit in $k$, with a homomorphism $\alpha$ from $\Gamma$ to $\mathrm{Aut}(T)$. So $\Gamma$ acts on $A$ by $(\gamma\cdot f)(t) = f(\gamma^{-1}(t))$, for $\gamma\in \Gamma, f\in\OO(T), t\in T$, and we can form the smash product $ H=A\# \Gamma.$ This is a Hopf algebra with the given coproduct of $A$ and with $\Gamma$ consisting of group-likes; and clearly $A \subseteq H \in\mathcal{CSC}(k)$. More generally, $k\Gamma$ can be replaced by any semisimple and cosemisimple Hopf algebra $\overline{H}$, with a Hopf algebra homomorphism $\alpha$ from $\overline{H}$ to $k\mathrm{Aut}(T)$. At one extremity $\alpha(\overline{H})$ could be $k 1_{\mathrm{Aut}(T)}$, yielding the tensor product  $H = A \otimes_k \overline{H}.$ A second large collection of examples is provided in $\S$\ref{abfingroup} by those group algebras $H = kG$ where $G$ has a finitely generated abelian normal subgroup $N$ of finite index in $G$, such that $G/N$ has no elements of order $\mathrm{char}\, k$. The main result of this section, Theorem \ref{semi}, suggests that these examples may go some way towards exhausting all the possibilities for such $H$. 

Before proving Theorem \ref{semi} we need to recall a concept from noncommutative ring theory, and a basic result on finite dimensional Hopf algebras \cite{Sk07}, extending \cite{Mas92}. An ideal $I$ of a noetherian ring $R$ is \emph{polycentral} if there are elements $x_1, \ldots , x_t \in I$ with $I = \sum_{i=1}^t x_i R$, such that $x_1$ is in the centre $Z(R)$ of $R$ and, for $j = 2, \ldots , t$, $x_j + \sum_{i=1}^{j-1} x_i R \in Z(R/\sum_{i=1}^{j-1} x_i R)$. Polycentral ideals share many of the properties of ideals of commutative noetherian rings, \cite[Chapter 4, $\S$2]{McRob}, \cite[Chapter 11, $\S$2]{Pa}.

\begin{Prop}\label{Skrysemisimple} {\rm(Skryabin, \cite{Sk07})} Every coideal subalgebra of a finite-dimensional semisimple Hopf $k$-algebra is semisimple.
\end{Prop}
\begin{proof}
Every left or right coideal subalgebra of a finite-dimensional Hopf algebra is Frobenius \cite[Theorem 6.1]{Sk07}. Since $k$ is algebraically closed, a coideal subalgebra of a semisimple Hopf algebra is Frobenius if and only if it is semisimple by \cite[Theorem 2.1]{Mas92}.
\end{proof}

Given the length of the statement and proof of Theorem \ref{semi}, the reader may find it helpful to look first at the discussion immediately following its proof, including Corollary \ref{primesemis}. The latter records the simplifications which occur if $A \subseteq H \in \mathcal{CSC}$ and $H$ is \emph{prime}.

\subsection{Structure theorem}\label{structure}

\begin{Thm}\label{semi}
Let $A \subseteq H \in\mathcal{CSC}(k)$ with $\GKdim(H) = n$. Let $N(A)$ denote the nilradical of $A$ and let $P$ be the unique minimal prime ideal of $A$ with $P \subseteq A^+$, as in Proposition \ref{Hprime}.  Let $G^{\circ} \subseteq G$ be the algebraic groups such that $\mathcal{O}(G) = A/N(A)$ and $\mathcal{O}(G^{\circ}) = A/P$, as in Remark \ref{group}.

\begin{enumerate}
\item $N(A)$ and $P$ are (left and right) $\overline{H}$-stable Hopf ideals of $A$, and $ N(A)H = N(H)$ and $PH$ are semiprime Hopf ideals of $H$.

\item $A/N(A) \subseteq H/N(A)H$ and $A/P \subseteq H/PH$ are in $\mathcal{CSC}(k)$, with Gelfand-Kirillov and global dimensions $n$. They are faithfully flat $\overline{H}$-Galois extensions of $A/N(A)$ and $A/P$ respectively.

\item Let $Q_1, \ldots , Q_t$ be the minimal prime ideals of $H$. Precisely one minimal prime of $H$, say $Q_1$, is contained in $H^+$, and this minimal prime contains $PH$. Reorder the remaining $Q_i$ and fix $r$, $1 \leq r \leq t$, so that $P \subseteq Q_i$ if and only if $i \leq r$. Then
$$ N(A)H = \bigcap_{i=1}^t Q_i, \qquad  PH = \bigcap_{i=1}^r Q_i, \qquad Q_j \cap A = P, \quad (1 \leq j \leq r).$$
Moreover
$$ H/N(A)H \cong \bigoplus_{i=1}^t H/Q_i \quad \textit{ and } \quad H/PH \cong \bigoplus_{i=1}^r H/Q_i, $$
direct sums of prime algebras of Gelfand-Kirillov and global dimensions $n$.

\item There are subalgebras 
\begin{equation}\label{chain} C\subseteq B \subseteq A \subseteq D \subseteq H,
\end{equation}
such that:
\begin{enumerate}
\item[(i)] $C:=A^{\co A/N(A)}$ is a local Frobenius left coideal subalgebra of $A$ with $C^+A\subseteq N(A)$, and $A$ is a free left and right $C$-module. Moreover, $C$ is invariant under the left adjoint action of $H$.

\item[(ii)] $B := A^{\co A/P}$ is a left coideal subalgebra of $A$, over which $A$ is flat and such that 
\begin{equation}\label{prune} P = B^+ A + N(A).
\end{equation}

\item[(iii)] There is a factor group algebra $k \Gamma$ of $\overline{H}$ with Hopf epimorphism $\alpha: H \to k\Gamma$, such that the left and right adjoint actions of $\overline{H}$ on $A/P$ both factor through an inner faithful $k\Gamma$-action. Thus, $D := H^{\co k\Gamma}$ is a left coideal subalgebra of $H$ with $H$ left and right faithfully flat over $D$. Moreover, $D$ is invariant under the left adjoint action of $H$, $D^+ H = H D^+ = \ker\alpha$ is a Hopf ideal of $H$ and 
\begin{equation}\label{kGamma} H/D^+ H \cong k \Gamma.
\end{equation}
\end{enumerate}

\item $D/A^+ D $ is semisimple.

\item For all $a \in A$ and $d \in D$, $ ad - da \in PD$, so that 
$$A/P \subseteq Z(D/PD),$$ where $Z(R)$ denotes the centre of the ring $R$.

\item $D/N(A)D$ has global dimension $n$, so it is homologically homogeneous and is a direct sum of prime algebras, each of GK-dimension and global dimension $n$.

\item There is a unique minimal prime ideal $L$ of $D$ with $L \subseteq D^+$. Moreover,
\begin{equation}\label{caught} L \cap A = P
\end{equation}
and \begin{equation}\label{Lformula} L=\bigcap_{i \geq 1} (D^+)^i + N(A)D.
\end{equation}

\item $D/L$ is an affine commutative domain of global dimension $n$.

\item $L$ is a left $H$-stable ideal of $D$, so $LH$ is a Hopf ideal of $H$. Thus, 
\begin{enumerate}
\item[(i)] $D/L$ is a left coideal subalgebra of $H/LH$ and a finite module over the Hopf subalgebra $A/P$ of $H/LH$; 
\item[(ii)] the left adjoint action of $H$ on $D$ induces a left adjoint action of $H/LH$ on $D/L$; and this action factors through an inner faithful group action, for some group $\Lambda$ which maps surjectively onto $\Gamma$.
\end{enumerate}

\item There are inclusions $$ N(A)H\subseteq PH \subseteq LH \subseteq Q_1, \quad \textit{ with } \quad Q_1 \cap D = L. $$
%

\item Let $E:=H^{\co H/LH}$, a left coideal subalgebra of $H$ with $ B \subseteq E \subseteq D.$ Then $E = D^{\co D/L}$ and it is invariant under the left adjoint action of $H$.

\item Assume further that $H$ is pointed. Then, $H$ is a faithfully flat left and right $E$-module, $Q_1 = LH$ and $H$ and $H/LH$ are crossed products:
$$ H\cong A\#_\sigma \overline{H} \quad \textit{ and } \quad H/Q_1 \cong (D/L)\#_\tau \Gamma, $$ 
for cocycles $\sigma$ and $\tau$. 
\end{enumerate}
\end{Thm}
\begin{proof} (1),(2) Since $\overline{H}$ is cosemisimple, $A$ is $\overline{H}$-orbitally semisimple by Theorem \ref{orbhold}(3), hence $N(A)$ is $\overline{H}$-stable by Proposition \ref{(H)hold}. In particular, $N(A)H$ is a nilpotent Hopf ideal of $H$ by Proposition \ref{Hprime}(2). It follows from Proposition \ref{gldimcrit}(1) that $H/N(A)H$ is smooth, so that $N(A)H$ is semiprime by Theorem \ref{smooththm}(3). Since $N(A)H$ is nilpotent, it is thus the nilradical of $H$. 

Since $N(A)$ is $\overline{H}$-stable, $P$ is also $\overline{H}$-stable and $PH$ is a Hopf ideal of $H$ by Proposition \ref{Hprime}(2). That $PH$ is semiprime follows for the same reasons as applied to $N(A)H$. The fact that the global and Gelfand-Kirillov dimensions of $H/N(A)H$ and $H/PH$ equal $n$ is a consequence of Theorem \ref{injthm}(3) coupled with Theorem \ref{smooththm}(1). The relevant extensions are faithfully flat $\overline{H}$-Galois by Theorem \ref{properties}(7)(ii)(a).

\medskip

\noindent(3) As already noted, Theorem \ref{smooththm}(3) applies to both $H/N(A)H$ and to $H/PH$, so that both of these algebras are finite direct sums of prime algebras of GK-dimension $n$. Therefore $\GKdim (H/Q_i) = n$ for all the minimal primes $Q_1, \ldots , Q_t$ of $H$, and the displayed intersections and direct sum decompositions are clear. The remainder of (3) follows from Proposition \ref{Hprime}(4).

\medskip

\noindent (4)(i) This was proved in Lemma \ref{commfacts}(2) and Proposition \ref{Hprime}(3).

(ii) This was proved in Lemma \ref{lemmaB}. 

(iii) Since $P$ is left and right $\overline{H}$-stable by Proposition \ref{Hprime}(2), the $\overline{H}$-actions on $A$ restrict to $\overline{H}$-actions on $A/P$. Consider first the \emph{right} adjoint action of $\overline{H}$ on $A/P$. By \cite[Theorem 2]{Sk17} it factors through a group algebra $k \Gamma$, via a Hopf epimorphism $\overline{H}\twoheadrightarrow k\Gamma$, with $k\Gamma$ acting inner faithfully on $A/P$. Let $\alpha$ be the Hopf epimorphism $H\twoheadrightarrow \overline{H}\twoheadrightarrow k\Gamma$ and $D := H^{\co \alpha}$, so that $D$ is a left coideal subalgebra of $H$, and is invariant under the left adjoint action of $H$ by \cite[Lemma 3.4.2(2)]{Mont}. Since $\alpha$ factors through $\pi:H \to \overline{H}$,
$$ A \subseteq H^{\co \pi} \subseteq H^{\co \alpha} = D.$$
By \cite[Corollary 1.5]{MulSchn}, $H$ is left and right faithfully flat over $D$, and $\ker\alpha = D^+ H$, whence $D^+H$ is a Hopf ideal and (\ref{kGamma}) follows. By Koppinen's lemma \cite[Lemma 1.4]{MulSchn} and the fact that $D^+H$ is an ideal of $H$,  $S(D^+ H) = HD^+ \subseteq D^+ H$. Since $S(A^+ H) = H A^+$, $S$ induces a bijection on the finite dimensional space $\overline{H}$, so that $HD^+ = D^+ H$. 

Now repeat the above argument for the \emph{left} adjoint action, yielding an epimorphism of Hopf algebras $\beta: H \twoheadrightarrow \overline{H} \twoheadrightarrow k\Lambda$, for a finite group $\Lambda$ with $k\Lambda$ acting inner faithfully on $A/P$.
We claim that 
\begin{equation}\label{harp}  \ker\alpha = \ker \beta, \quad \textit{ so that } \quad \Gamma = \Lambda.
\end{equation}
Let $h \in \ker\alpha$ and $v \in A/P$. By the proof of Lemma \ref{leftrightorbss}(1),
$$ 0 = ad_r(h)(Sv) = S(ad_{\ell}(S^{-1}h)(v)). $$
Since $S$ is an automorphism of $A/P$, this implies that $ ad_{\ell}(S^{-1}h)(A/P) = 0, $ so that 
$$ S^{-1}(\ker\alpha) \subseteq  \ker\beta. $$
A similar argument yields the reverse inclusion. But $S(D^+H)=D^+H$ as above, hence $\ker\beta=S^{-1}(D^+H)=D^+H=\ker\alpha$, proving (\ref{harp}).
\medskip

\noindent (5) Note that $D/A^+D = D/(A^+H \cap D) \cong (D + A^+H)/A^+ H \subseteq \overline{H}$, where the equality follows by faithful flatness of $H$ over $D$. Therefore (5) is a consequence of Proposition \ref{Skrysemisimple}, since $\overline{H}$ is semisimple.

\medskip

\noindent (6) Since $D^+ \subseteq \ker \alpha $, $D$ acts by $\epsilon$ in the right adjoint action of $H$ on $A/P$.
Thus, for $a\in A, d\in D$, 
$$ ad=\sum d_1S(d_2)ad_3 = \sum d_1\, (ad_r)(d_2)(a) \equiv \sum d_1\epsilon(d_2)a = da \mod PH, $$ 
since $D$ is a left coideal. By (4)(iii), $D \subseteq H$ is faithfully flat, so $PH \cap D = PD$.

\medskip

\noindent (7) Since $H$ is right faithfully $D$-flat, $N(A)H\cap D=N(A)D$. Let $A'=A/N(A), D'=D/N(A)D$ and $H'=H/N(A)H$.

First, $\gldim (D') \leq n$. To see this, note that $\gldim(H') = n$ by (1) and $H'/(D')^+ H'$ is cosemisimple by (\ref{kGamma}), so the inequality follows by \cite[Lemma 9]{Uli}, (in which the key point is that the cosemisimplicity of $H'/(D')^+ H'$ ensures that the left $D'$-module direct sum decomposition $H' = D' \oplus U$ of \cite[Corollary 2.9]{MW} can be achieved as $D'$-bimodules). 

As just explained, $D'$ is a left $D'$-direct summand of $H'$, so $D'$ is a projective $A'$-module by Theorem \ref{properties}(7)(c). Let $V$ be an irreducible left $D'$-module, and suppose that $\prdim_{D'} (V) = t$. Restricting a $D'$-projective resolution of $V$ to $A'$, it follows that $\prdim_{A'} (V) \leq t$. However, $\mathrm{dim}_k (V) < \infty$ and $A'$ is a finite direct sum of commutative affine domains, each of global dimension $n$, by (the commutative case of) Theorem \ref{smooththm}(3). Hence all the irreducible $A'$-modules, and so all the finite dimensional $A'$-modules, have projective dimension $n$. Thus $n \leq t$, so $n = t$ and $D'$ is homologically homogeneous. The direct sum decomposition of $D'$ now follows from \cite[Theorem 5.3]{BrH2}.

\medskip

\noindent (8) There is a unique minimal prime ideal $L$ of $D$ with $L \subseteq D^+$ by the decomposition of $D/N(A)D$ into a direct sum of prime rings in (7). If $P = N(A)$ then clearly $P \subseteq L$ since $N(A)D = N(A)H \cap D$ is a nilpotent ideal of $D$ by (1). Suppose on the other hand that $N(A) \subsetneq P$. Since $P$ is the unique minimal prime of $A$ with $P \subseteq A^+$, there exists $y \in A \setminus A^+$ with $yP \subseteq N(A)$. By stability of $P$,
$$ yDP \subseteq yPD \subseteq N(A)D \subseteq L.$$
Hence, since $y \notin L$, 
\begin{equation}\label{loop} P \subseteq L \cap A. \end{equation}
For the reverse inclusion, note that $D/L$ is a finite (left, say) $A/L \cap A$-module, so that 
$$\GKdim (A/L \cap A) = \GKdim(D/L) = n$$
by \cite[Proposition 5.5]{KL} and (7). This forces equality to hold in (\ref{loop}), as $\GKdim (A/P) = n$ by (2), and a proper factor of an affine commutative domain has strictly lower GK-dimension \cite[Proposition 3.15]{KL}.

We now prove (\ref{Lformula}). Since $L \subseteq D^+$ and $L/N(A)D$ is generated by an idempotent in $D/N(A)D$ by (7), 
$$ L/N(A)D \subseteq (\bigcap_{i \geq 1}(D^+)^i + N(A)D)/N(A)D.$$
Hence $L\subseteq \bigcap_{i \geq 1}(D^+)^i + N(A)D$. We prove the reverse inclusion. Since $D/A^+D$ is semisimple by (5) and the image of $A$ is central in $D/L$ by (6) and (\ref{caught}), $D^+/L$ is a polycentral ideal of $D/L$. Hence, by the version of Krull's Intersection Theorem for polycentral ideals, \cite[Theorems 11.2.8 and 11.2.13]{Pa}, $$ \bigcap_{i \geq 1} (D^+)^i \subseteq L, $$ since $D/L$ is a prime ring. This proves the required equality.

\medskip

\noindent (9) We first prove that $D/L$ is commutative. By (\ref{Lformula}), it suffices to prove that $D/(D^+)^i$ is commutative for each $i \geq 1$. Choose elements $ a_1,\ldots,a_m$ of  $A^+$ whose images form a $k$-basis of $A^+/(A^+)^2$. Let $e\in D^+$ be such that $e+A^+D$ is the central idempotent generator of $D^+/A^+D$ guaranteed by (5). Then
$$ D^+/(D^+)^2 = (De + \sum_j Da_j + (D^+)^2)/(D^+)^2 = (ke + \sum_j ka_j + (D^+)^2)/(D^+)^2. $$
Since the quotient $D^+/(\sum_j ka_j + (D^+)^2)$ is a factor of $D^+/A^+D$, it is idempotent and, being also a factor of $D^+/(D^+)^2$, it must be zero, so  
\begin{equation}\label{spanner} D^+=\sum_j ka_j + (D^+)^2.
\end{equation}
Therefore, for each $i\geq 2$, 
\begin{equation}\label{span} D^+/(D^+)^i \textit{ is spanned by monomials of length at most } i-1 \textit{ in } a_1,\ldots,a_m.
\end{equation} 
Hence $D/(D^+)^i$ is commutative, as required. Thus $D/L$ is an affine commutative domain, since $L$ is prime and $D$ is a finite $A$-module. That the global dimension of $D/L$ is $n$ follows from (7) and the fact that $L$ is a minimal prime.
\medskip

\noindent (10) First note that $N(A)$ and $D^+$ are left $H$-stable by (1) and (4)(iii) respectively, hence $L$ is left $H$-stable by (\ref{Lformula}). Moreover, $N(A)H$ and $D^+H$ are Hopf ideals of $H$ again by parts (1) and (4)(iii). Thus $\cap_i(D^+H)^i = (\cap_i (D^+)^i)H$ is a Hopf ideal by Lemma \ref{powers}. Therefore, by (\ref{Lformula}), $$ LH = N(A)H + \bigcap_i(D^+H)^i $$ is a Hopf ideal of $H$.

\medskip

\noindent (i) By faithful flatness of $H$ over $D$, as ensured by (4)(iii),  
\begin{equation}\label{hike} LH \cap D = L.
\end{equation}
Therefore, $D/L$ is a left coideal subalgebra of $H/LH$. By (\ref{caught}), $A/P$ embeds in $H/LH$, and this map is a homomorphism of Hopf algebras whose image is contained in $D/L$. For the rest of the proof of (10), we write $A' := A/P$, $D':=D/L$ and $H':=H/LH$, so that $A' \subseteq D' \subseteq H'$ by (\ref{caught}) and (\ref{hike}), with $A'$ a Hopf subalgebra and $D'$ a left coideal subalgebra of $H'$, and with $D'$ an affine commutative domain by (9).

\medskip

\noindent(ii) By (4)(iii) the left adjoint action of $H$ on itself preserves $D$ and, since $L$ is left $H$-stable, this induces a left adjoint action of $H'$ on $D'$. Let $a \in A'^+$ and $d \in D'$. Then, by commutativity of $D'$ and the fact that $A'$ is a Hopf subalgebra, 
\begin{equation}\label{grab} ad_{\ell}(a')(d') = \sum a'_1 d' S(a'_2) = \sum a'_1 S(a'_2) d' = \epsilon(a') d'  = 0 
\end{equation}
for all $a' \in A'^+$ and $d' \in D'$. Thus this left adjoint action on $D'$ factors through $H'A'^+ = A'^+ H'$.

Since $H'/A'^+ H'$ is semisimple and cosemisimple and $D'$ is a commutative domain by (9), this action in turn factors through an inner faithful group action by \cite[Theorem 2]{Sk17}, say $H'/I\cong k\Lambda$ for some finite group $\Lambda$ and some Hopf ideal $I$ of $H'$ that annihilates $D'$ under the left adjoint action.
However, by (4)(iii) the left adjoint action of $H$ on $A'$ factors through $D^+H$ with $H/D^+ H \cong k\Gamma$ acting inner faithfully. And, since $A' \subseteq D'$, it follows that
\begin{equation}\label{beech} I \subseteq D'^+ H'
\end{equation} 
and we have a Hopf epimorphism $k\Lambda\twoheadrightarrow k\Gamma$. This proves (ii).

%


\medskip

\noindent (11) Only the parts involving $Q_1$ remain to be proved. That $LH \subseteq Q_1$ follows by essentially the same argument as applied to $P$ in (8) - namely, there exists $z\in D \setminus D^+$ with $zL = \{0 \}$, so $ z(HL) = z(LH) = \{0\} \subseteq Q_1,$ with $z \notin Q_1$, so that $L \subseteq Q_1\cap D$. And extending the argument along the lines of (8) yields $Q_1 \cap D = L$.

\medskip
\noindent (12) Since $LH$ is a Hopf ideal by (10), we may consider the coinvariants $E$ of the Hopf surjection $H\twoheadrightarrow H/LH$. On one hand, it follows from (\ref{Lformula}) that $LH\subseteq D^+H$, hence we must have $E = H^{\co H/LH} \subseteq H^{\co k\Gamma} = D$ by (\ref{kGamma}). On the other hand, since $D\subseteq H$ is faithfully flat, it follows from (\ref{caught}) that $LH\cap A = P$, hence $B=A^{\co A/P} \subseteq H^{H/LH} = E$. Furthermore, $E$ is invariant under the left adjoint action of $H$ by \cite[Lemma 3.4.2(2)]{Mont}.

\medskip
\noindent (13) Suppose that $H$ is pointed. By \cite[Theorem]{Mas91} $H$ is faithfully flat over its left coideal subalgebra $E$. As was pointed out in Remark \ref{propremarks}(4), by \cite[Corollary 4.3]{Schn92}, also stated as \cite[Theorem 8.4.8]{Mont}, $ H \cong  A \#_\sigma \overline{H}$, for some cocycle $\sigma$. Similarly, $H/LH$ is pointed, so the inclusion $D/L \subseteq H/LH$ yields the decomposition 
$$H/LH \cong (D/L)\#_\tau \Gamma.$$ 
Moreover, this crossed product by a finite group acting faithfully on the commutative domain $D/L$ is prime by \cite[Corollary 12.6]{Pa2}.
\end{proof}

\begin{Rmk}\label{picture} Consider again, in the light of the above theorem, the exact sequence of algebras suggested in \cite[$\S$ 6]{LWZ}, and discussed here at the start of $\S$\ref{prime2} and in Remark \ref{primermks}(2). One might hope that the true picture underlying Theorem \ref{semi} consists of a sequence of algebras
$$ 1 \longrightarrow E \longrightarrow D \longrightarrow H \longrightarrow k\Gamma \longrightarrow 1,$$
where $E \subseteq D \subseteq H$ and $k\Gamma \cong H/D^+ H$, with $H$ and $k\Gamma$ Hopf algebras, $E$ and $D$ left adjoint invariant left coideal subalgebras of $H$ with $E$ finite dimensional, with
$$ D/E^+ D \textit{ an affine commutative domain, left coideal subalgebra of } H/E^+ H,$$
and 
$$ H/E^+ H \cong (D/E^+ D) \#_{\tau} k \Gamma, $$
a prime crossed product. 

That is, in crude terms, such a Hopf algebra $H$ in $\mathcal{CSC}(k)$ would be
$$ \textit{(finite dimensional)-by-(prime crossed product)}$$
where the prime crossed product is of a commutative domain acted on by a finite group algebra. Obstructions to confirming this description include the absence of positive answers to the following, where we retain the notation of Theorem \ref{semi}:

\begin{Qtns}\label{stuck} (1) Is $E$ a finite dimensional algebra?

\noindent (2) Is $H/LH$ a prime crossed product, with $LH = Q_1$ and $H/LH \cong (D/L)\#_{\tau} \Gamma$, if $H$ is not pointed?
\end{Qtns}
\end{Rmk}

\medskip

When $H$ is prime, the first of these issues disappears, yielding:

\begin{Cor}\label{primesemis} Let $A \subseteq H \in\mathcal{CSC}(k)$, with $H$ prime.  Then, after replacing $A$ by a larger smooth commutative affine domain $D$ which is a left $H$-invariant left coideal subalgebra of $H$,
\begin{enumerate}
\item $H/D^+ H \cong k\Gamma$ for a finite group $\Gamma$ whose order is a unit in $k$;
\item There exists a group $\Lambda$ which acts faithfully on $D$ via the left adjoint action; and the group algebra $k\Lambda$ maps surjectively onto $k\Gamma$.
\item Suppose in addition that $H$ is pointed. Then $H$ is a crossed product of $D$ by $k \Gamma$, that is, $H\cong D\#_\sigma k\Gamma$ for some cocycle $\sigma$.
\end{enumerate}
\end{Cor}
\begin{proof}
The hypothesis that $H$ is prime implies that, in the notation of Theorem \ref{semi}, $Q_1 = \{0\}$. Thus, $P =L =\{0\}$ by Theorem \ref{semi}(11) and (\ref{caught}). Thus $D$ is a commutative affine domain containing $A$, and the corollary is a special case of the theorem.
\end{proof}

\subsection{Examples and consequences}\label{excon} The first of the following two simple examples shows that the inclusions of (\ref{chain}) can all be strict. The second shows that even when $H$ is prime $D$ may strictly contain $A$.

\begin{Exs}\label{fingroup} Let $k$ be algebraically closed of characteristic 0.

\noindent (1) Let $G = (\langle x \rangle \times S_3) \rtimes C_2,$ where $S_3$ is the symmetric group on 3 symbols and $\langle x \rangle$ is the infinite cyclic group. Let $\sigma$ and $\beta$ be respectively a 3-cycle and a 2-cycle in $S_3$ and let $a$ be a generator of $C_2$, with $C_2$ acting trivially on $S_3$ and acting on $\langle x \rangle$ by $a\cdot x = x^{-1}$. Let $H = kG$ and $A = k(\langle x \rangle \times \langle \sigma \rangle)$. Thus $A \subseteq H \in \mathcal{CSC}(k)$, with $\overline{H}=k(\langle \beta \rangle \times C_2)$. 

Then, $$ P = (\sigma - 1)A, \qquad B = k\langle \sigma \rangle. $$ The $H$-action on $A/P\cong k\langle x\rangle$ factors through $\Gamma \cong C_2$, so $$ D = k(\langle x \rangle \times S_3), \qquad L = PD + (\beta-1)D. $$
%
%
Moreover, $D/L \cong k\langle x \rangle$ and $E=kS_3$.
Thus in this case $$ C \subsetneq B \subsetneq A \subsetneq D \subsetneq H, \qquad B\subsetneq E\subsetneq D
\quad \text{ and } \quad PH \subsetneq LH = Q_1 = E^+H, $$
with $$ H \cong A\# \overline{H} \cong D \# \Gamma \quad \text{ and } \quad A/P \cong D/L \cong k \langle x \rangle. $$

\medskip

\noindent (2) Let $H = k[x^{\pm 1}]$ with $x$ grouplike, and let $A = k[x^{\pm 2}]$. Then $A \subseteq H \in \mathcal{CSC}(k)$, with $\overline{H} = kC_2$. In this case $D = H$ and $\Gamma = \{1\}$, but $A \subsetneq D$.

\end{Exs}

\medskip

Bringing together the description of prime algebras in $\mathcal{CSC}(k)$ from Corollary \ref{primesemis} with the Clifford-theoretic analysis of Theorem \ref{bound} yields 

\begin{Thm}\label{dims} Let $A \subseteq H \in \mathcal{CSC}(k)$, with $H$ prime, and let $D$ and $\Gamma$ be as in Corollary \ref{primesemis}. Let $V$ be a simple left $H$-module, and choose $\mathfrak{m} \in \mathrm{Maxspec}(D)$ with $\mathrm{Ann}_V(\mathfrak{m}) \neq \{0\}$.
\begin{enumerate}
\item $|\Gamma : C_{\Gamma}(\mathfrak{m})| \leq \mathrm{dim}_k (V) \leq |\Gamma |,$ where $C_{\Gamma}(\mathfrak{m}) := \{\gamma \in \Gamma : \mathfrak{m}^{\gamma} = \mathfrak{m}\}$.
\item $\mathrm{PI-degree}(H) = |\Gamma |.$ In particular, the maximum dimension of simple $H$-modules is $\Gamma$, and the simple $H$-modules of dimension $|\Gamma|$ are induced from simple $D$-modules.
\item If $H$ is pointed then (1) can be strengthened: there exists $\ell \leq |C_{\Gamma}(\mathfrak{m})$ such that 
$$ \mathrm{dim}_k (V) = \ell |\Gamma : C_{\Gamma}(\mathfrak{m})|,$$
with $V$ a free $A/\mathfrak{m}^{\Gamma}$-module of rank $\ell$.
\end{enumerate}
\end{Thm}

\begin{proof} (1) The existence of $\mathfrak{m}$ follows as in the proof of Theorem \ref{bound}(1), with $A$ replaced by $D$, noting that in the present situation $\overline{H}$ is replaced by $k \Gamma$, and 
$$ \mathfrak{m}^{(\overline{H})} = \mathfrak{m}^{H} = \bigcap\{ \mathfrak{m}^{\gamma} : \gamma \in \Gamma \}. $$
Since $D/\mathfrak{m}^{\Gamma}$ is a direct sum of copies of $k$, an easier version of the proof of Theorem \ref{bound}(2) shows that $D/\mathfrak{m}^{\Gamma}$ embeds in $V$ as a $D$-submodule, proving the first inequality in (1).

By Theorem \ref{semi}(4)(iii) and Corollary \ref{primesemis}(1), $H$ is a locally free module of rank $|\Gamma|$ over the affine commmutative domain $D$, so that
$$ \mathrm{dim}_k (H/H\mathfrak{m}) = |\Gamma|. $$
Since the $H$-module $V$ is a factor of $H/H\mathfrak{m}$, (1) is proved.

\medskip

\noindent (2) Since $\Gamma$ acts on $D$ by $k$-algebra automorphisms, the extension $Q(D^{\Gamma}) \subseteq Q(D)$ is Galois with Galois group $\Gamma$. Hence $\mathrm{dim}_{Q(D^{\Gamma})} Q(D) = |\Gamma|$ by Galois theory. By generic freeness, \cite[Theorem 14.4]{Ei}, there is a non-empty open subset $\mathcal{P}$ of $\mathrm{Maxspec}(D^{\Gamma})$ such that, for $\mathfrak{p} \in \mathcal{P}$, $D_{\mathfrak{p}}$ is $D^{\Gamma}_{\mathfrak{p}}$-free of rank $|\Gamma|$. Because the field extension $Q(D^{\Gamma} \subseteq D$ is Galois, it is separable. Moreover $D$ is smooth by Theorem \ref{semi}(9). Therefore, by \cite[Theorem 7 of Ch.II, $\S$5]{Sh}, the inclusion $D^{\Gamma} \subseteq D$ is generically unramified. That is, there exists $\mathfrak{p} \in \mathcal{P}$ such that $\mathfrak{p}D$ is semimaximal with $\mathrm{dim}_k (D/\mathfrak{p}D) = |\Gamma|$. Let $\widehat{\mathfrak{m}} \in \mathrm{Maxspec}(D)$ with $\mathfrak{p} \subseteq \widehat{\mathfrak{m}}$. By classical invariant theory, \cite[Chap.1]{Be},
$$ \Gamma-\mathrm{orbit} (\widehat{\mathfrak{m}}) = \{ \mathfrak{m} \in \mathrm{Maxspec}(D) : \mathfrak{p} \subseteq \mathfrak{m}\}.$$ 
Thus $| \Gamma-\mathrm{orbit} (\widehat{\mathfrak{m}})| = |\Gamma|$. Therefore (1) implies (2).

\medskip
\noindent (3) Suppose that $H$ is pointed. Then $H $ is a crossed product by Corollary \ref{primesemis}(3), so there are units $\widehat{\gamma}_i$ in $H$ such that $H = \oplus_i D\widehat{\gamma}_i$, and it follows easily that, for all $i$,
$$ \mathrm{Ann}_V(\mathfrak{m}^{\widehat{\gamma}_i}) = \widehat{\gamma}_i \mathrm{Ann}_V(\mathfrak{m}). $$
Moreover, $V$ is a semisimple $A/\mathfrak{m}^{\Gamma}$-module and hence 
$$ V = \bigoplus_i \mathrm{Ann}_V(\mathfrak{m}^{\widehat{\gamma}_i}) . $$
Setting $\ell = \mathrm{dim}_k (\mathrm{Ann}_V(\mathfrak{m}))$ therefore gives (3).
\end{proof}

\section*{Acknowledgements} Some of this research will form part of the PhD thesis of the second author at the University of Glasgow. His PhD is funded by the Portuguese Foundation for Science and Technology Fellowship SFRH/BD/102119/2014, to whom we are very grateful. The research of the first author was supported in part by Leverhulme Emeritus Fellowship EM-2017-081. 

We thank Stefan Kolb (Newcastle), Uli Kr\"{a}hmer (Dresden), Christian Lomp (Porto) and Xingting Wang (Howard) for very helpful comments.

\end{document}